\documentclass[12pt]{amsart}
\usepackage{amscd, amsfonts, amssymb}
\usepackage{epsfig}
\usepackage[mathscr]{eucal}
\usepackage{verbatim}
\usepackage{fullpage}
\usepackage{latexsym}
\usepackage{lscape}
\usepackage{enumerate}

\newcommand\cc{\mathfrak c}
\renewcommand\gg{\mathfrak g}
\newcommand\hh{\mathfrak h}
\newcommand\kk{\mathfrak k}
\newcommand\frakl{\mathfrak l}
\newcommand\mm{\mathfrak m}

\newcommand\pp{\mathfrak p}

\newcommand\mb{\mathbb}

\newcommand{\inprod}[1]{\langle #1\rangle}

\newcommand\inverse{{^{-1}}}

\newcommand{\ZZ}{\mathbb{Z}}
\newcommand{\NN}{\mathbb{N}}

\newcommand\gl{\mathfrak{gl}}

\newcommand\ra{\rightarrow}

\DeclareMathOperator{\Ad}{Ad}

\DeclareMathOperator{\Char}{char}

\DeclareMathOperator{\diag}{diag}

\DeclareMathOperator{\GL}{GL}

\DeclareMathOperator{\SL}{SL}

\DeclareMathOperator{\Lie}{Lie}
\DeclareMathOperator{\Mat}{Mat}

\DeclareMathOperator{\End}{End}

\newcommand{\tuple}[1]{{\mathbf {#1}}}

\numberwithin{equation}{section}

\newtheorem{thm}[equation]{Theorem}

\newtheorem{lem}[equation]{Lemma}
\newtheorem{cor}[equation]{Corollary}
\newtheorem{prop}[equation]{Proposition}

\theoremstyle{definition}
\newtheorem{defn}[equation]{Definition}
\newtheorem{exmp}[equation]{Example}

\theoremstyle{remark}
\newtheorem{rem}[equation]{Remark}

\theoremstyle{remark}
\newtheorem{rems}[equation]{Remarks}

\newcommand{\ovl}[1]{{\overline{#1}}}


\title[Complete Reducibility and Conjugacy classes of tuples]
{Complete Reducibility and Conjugacy classes of tuples
in Algebraic Groups and Lie algebras}

\author[M.\  Bate]{Michael Bate}
\address
{Department of Mathematics,
University of York,
York YO10 5DD,
United Kingdom}
\email{meb505@york.ac.uk}

\author[B.\ Martin]{Benjamin Martin}
\address
{Mathematics and Statistics Department,
University of Canterbury,
Private Bag 4800,
Christchurch 1,
New Zealand}
\email{B.Martin@math.canterbury.ac.nz}

\author[G. R\"ohrle]{Gerhard R\"ohrle}
\address
{Fakult\"at f\"ur Mathematik,
Ruhr-Universit\"at Bochum,
D-44780 Bochum, Germany}
\email{gerhard.roehrle@rub.de}

\author[R.\ Tange]{Rudolf Tange}
\address
{Department of Mathematics,
University of York,
York YO10 5DD,
United Kingdom}
\email{rht502@york.ac.uk}

\begin{document}

\begin{abstract}
Let $H$ be a reductive subgroup of a reductive group $G$ over an algebraically
closed field $k$. We consider the action of $H$ on $G^n$,
the $n$-fold Cartesian product of $G$
with itself, by simultaneous conjugation.
We give a purely algebraic characterization of the closed
$H$-orbits in $G^n$, 
generalizing work of Richardson
which treats the case $H = G$.

This characterization turns out to be a natural generalization of Serre's
notion of $G$-complete reducibility.
This concept appears to be new, even in characteristic zero.
We discuss how to extend some key results on $G$-complete
reducibility in this framework.  We also consider some rationality questions.
\end{abstract}

\maketitle

\setcounter{tocdepth}{1}
\tableofcontents

\section{Introduction}
\label{sec:intro}

Let $G$ be a reductive linear algebraic group over an algebraically
closed field, and suppose $G$ acts on an affine variety $V$.
A fundamental problem in geometric invariant theory is to
determine the closed orbits of $G$ in $V$.  These orbits
correspond to the points in the quotient variety $V/\!\!/G$,
so this is the first step towards understanding the geometry
of the quotient.  Often it is of particular interest to find
the open subset of stable orbits, which consists of points
on which the quotient map $\pi\colon V\ra V/\!\!/G$ is
especially well-behaved.  Moreover, once the closed
orbits are known, one can study degeneration phenomena:
the way in which a point in a non-closed orbit can be brought
inside a closed orbit by taking a limit along a cocharacter.

An important family of examples arises as follows.
Take $G$ to be a subgroup of ${\rm GL}(W)$ for some
finite-dimensional vector space $W$, choose a
subvariety $C$ of ${\rm End}(W)$ that is stable
under conjugation by $G$, and take $V$ to be
$C^n$ for some $n\in \NN$, where $G$ acts on
$C^n$ by simultaneous conjugation.  Typically
$C$ carries some algebraic structure:
it might be a subgroup of ${\rm GL}(W)$, or a Lie subalgebra
or associative subalgebra of ${\rm End}(W)$.
For instance, if $G= {\rm GL}(W)$ and $C= {\rm End}(W)$,
then, by work of H.\ Kraft
(\cite[Prop.\ 4.4]{kraft1} or \cite[II.2.7 Satz 2]{kraft}),
for any $v= (c_1,\ldots, c_n)\in C^n = V$,
the orbit $G\cdot v$ is closed if and only if $W$ is semisimple
as an $A$-module, where $A$ is the associative
subalgebra of ${\rm End}(W)$ generated by $c_1,\ldots, c_n$;
moreover, if $G\cdot v$ is not closed,
then the degeneration process referred to
above is the ``semisimplification'',
in which one replaces the $A$-module
$W$ with the direct sum of its composition factors
(\cite[Prop.\ 4.5]{kraft1} or \cite[II.2.7 Satz 3]{kraft}).

Now consider the case when $C= G$.
In his seminal work \cite[Thm.\ 16.4]{rich},
Richardson gave an algebraic characterization
of the closed $G$-orbits in $G^n$.  In \cite[Thm.\ 3.1]{BMR}
it was shown that his criterion for an orbit to be
closed can be formulated using the representation-theoretic
notion of $G$-complete reducibility due to Serre, \cite{serre1.5}.
This concept has been much studied
and it has proved a useful tool for exploring
the subgroup structure of simple algebraic groups \cite{liebeckseitz}, \cite{liebeckseitz2}, \cite{liebecktesterman}, \cite{seitz1}, \cite{serre0}, \cite{serre1}, \cite{serre2}.
The approach to $G$-complete reducibility via geometric
invariant theory has proved very fruitful, cf.\
\cite{BMR}, \cite{BMR2}, \cite{BMRT}, and \cite{GIT}.

It is natural to extend Richardson's study and determine the closed $H$-orbits
in $G^n$ for an arbitrary reductive subgroup $H$ of $G$.  In this paper
we show that there is also an algebraic interpretation of the closed orbit
condition in this case.  We introduce the notion of a
\emph{relatively $G$-completely reducible subgroup of $G$ with respect to $H$}
(Section \ref{subsec:relative}) and prove the following result (see Theorem \ref{thm:orbitcrit}).

\begin{thm}
\label{thm:mainthm}
Let $H$ be a reductive subgroup of $G$.
Let $K$ be the algebraic subgroup of $G$ generated by elements
$x_1,\ldots,x_n \in G$.
Then
$H\cdot(x_1, \ldots,x_n)$ is closed in $G^n$
if and only if
$K$ is relatively $G$-completely reducible with respect to $H$.
\end{thm}

This generalizes Richardson's result \cite[Thm.\ 16.4]{rich}
which is the special case of Theorem~\ref{thm:mainthm} when $H = G$.

Note that we can embed $G$ in some ${\rm GL}(W)$, so this fits into the
general setting discussed above (take $C=G$ and replace $G$ with $H$).
In fact, we can take $G$ to be equal to ${\rm GL}(W)$ if we wish
(cf.~Corollary \ref{cor:GvsM}).

The definition of $G$-complete reducibility involves cocharacters of $G$.
Theorems about $G$-complete reducibility often involve taking not arbitrary
cocharacters of $G$ but cocharacters of a proper reductive subgroup
$H$ of $G$ (see \cite[Prop.\ 5.7]{BMRT}, for example).
The notion of relative complete reducibility
gives a systematic way to formalise such arguments: hence our results are
of interest even if one is concerned mainly with $G$-complete reducibility.

Armed with Theorem \ref{thm:mainthm},
we explore some basic properties
of relative $G$-complete reducibility in Section \ref{sec:relative}.
We study the extent to which results about $G$-complete reducibility
extend to relative $G$-complete reducibility, concentrating on what happens
when one varies $H$ or other parameters in the definition.
Returning briefly to the more general setting described above,
we extend the notion of relative $G$-complete reducibility to
certain Lie algebras and associative algebras associated to $G$:
for instance, we consider the case when $C$ is the associative
subalgebra of ${\rm End}(W)$ spanned by $G$.
We also characterize the $H$-stable orbits
in $G^n$ in terms of this notion of relative
$G$-complete reducibility (Proposition \ref{prop:rich-stable}); this
generalizes Richardson's result for the special case $H=G$,
\cite[Prop.\ 16.7]{rich}. From this
we deduce Theorem~\ref{thm:uniquegcr}, a group-theoretic
analogue of the fact that the closure of an $H$-orbit in $G^n$
contains a unique closed $H$-orbit. The other main results in
Section~\ref{sec:relative} are
Theorems~\ref{thm:optpar} and \ref{thm:relativenormal} and
Proposition~\ref{prop:SnormalizesK}.

In Section \ref{sec:rationality},
we define relative $G$-complete reducibility
over an arbitrary field.
We answer a generalization of
a question due to Serre about the behaviour of
$G$-complete reducibility under
separable field extensions (Theorem \ref{thm:serresquestion}).  We discuss the formalism of optimal destabilizing R-parabolic subgroups and give an application (Theorem \ref{thm:gcroptoverk}).
We finish the paper with a section containing a collection of examples.
In particular, we study the case of relative
$\GL(W)$-complete reducibility and give some characterizations in terms of the natural module $W$.

\section{Notation and preliminaries}
\label{sec:prelims}

\subsection{Basic notation}

Let $k$ be an algebraically closed field, and
let $H$ be a linear algebraic group defined over $k$.  All varieties are affine varieties over $k$ unless otherwise noted.
We let $Z(H)$ denote the centre of $H$ and $H^0$ the connected component of
$H$ that contains $1$.  By a subgroup of $H$ we mean a closed subgroup.
If $K$ is a subgroup of $H$, then $C_H(K)$ is the centralizer of $K$ in $H$
and $N_H(K)$ is the normalizer of $K$ in $H$.  We say that $H$ is
\emph{linearly reductive} if every rational representation
of $H$ is semisimple.

For the set of cocharacters (one-parameter subgroups) of $H$ we write $Y(H)$;
the elements of $Y(H)$ are the homomorphisms from
$k^*$ to $H$. There is an action of $H$ on $Y(H)$ given by
$(h\cdot \lambda)(a) = h\lambda(a)h^{-1}$ for
$\lambda\in Y(H)$, $h\in H$ and $a \in k^*$.

The \emph{unipotent radical} of $H$ is denoted $R_u(H)$; it is the maximal
connected normal unipotent subgroup of $H$.
The algebraic group $H$ is called \emph{reductive} if $R_u(H) = \{1\}$;
note that $H$ is reductive if and only if $H^0$ is reductive.

Throughout the paper, $G$ denotes a reductive algebraic group, possibly
non-connected.
We denote the Lie algebra $\Lie G$ by $\gg$, and likewise for subgroups of $G$.  We define $\Mat_m$ to be the associative algebra of $m\times m$ matrices over $k$.

Frequently, we consider the diagonal action of $G$
on $G^n$, the $n$-fold cartesian product of $G$ with itself,
by simultaneous conjugation:
$$
g\cdot(x_1,\ldots,x_n) := (gx_1g\inverse,\ldots,gx_ng\inverse),
$$
for all $g \in G$ and $(x_1,\ldots,x_n) \in G^n$.
Note that any subgroup $H$ of $G$ acts on $G^n$ in the same way.
We also consider the action of $G$ on $\gg^n$ by diagonal simultaneous
adjoint action.

Let $A$ be an algebraic group, a Lie algebra or an associative algebra.  If $n\in {\mathbb N}$ and $\tuple{x}= (x_1,\ldots, x_n)\in A^n$, then we say that $\tuple{x}$ generates $A$ if the $x_i$ generate $A$ as an algebraic group (resp.\ Lie algebra, resp.\ associative algebra).  By this we mean in the algebraic group case that the algebraic subgroup of $A$ generated by the $x_i$ is the whole of $A$, and we say that the algebraic group $A$ is topologically finitely generated.


\subsection{Non-connected reductive groups}
\label{subsec:noncon}
Since we want to work with reductive groups which are not necessarily connected,
we need to extend several familiar ideas from connected reductive groups.
The crucial ingredient of this extension is the introduction of so-called
\emph{Richardson parabolic subgroups} (\emph{R-parabolic subgroups}) of the
reductive group $G$.
We briefly recall the main definitions and results;
for more details and further results,
the reader is referred to \cite[Sec.\ 6]{BMR}.

\begin{defn}
\label{defn:rpars}
For each cocharacter $\lambda \in Y(G)$, let
$P_\lambda = \{ g\in G \mid \underset{a \to 0}{\lim}\,
\lambda(a) g \lambda(a)\inverse \textrm{ exists} \}$
(for the formal definition of such limits, see Definition \ref{defn:limit}).
Recall that a subgroup $P$ of $G$ is \emph{parabolic}
if $G/P$ is a complete variety.
The subgroup $P_\lambda$ is parabolic in this sense, but the converse
is not true: e.g.,\ if $G$ is finite,
then every subgroup is parabolic, but the only
subgroup of $G$ of the form $P_\lambda$ is $G$ itself.
We define
$L_\lambda = \{g \in G \mid \underset{a \to 0}{\lim}\,
\lambda(a) g \lambda(a)\inverse = g\}$.  Then $L_\lambda$ is a reductive subgroup of $G$ and we have $L_\lambda= C_G(\lambda(k^*))$ and $P_\lambda = L_\lambda \ltimes R_u(P_\lambda)$.
We also have
$R_u(P_\lambda) = \{g \in G \mid \underset{a \to 0}{\lim}\,
\lambda(a) g \lambda(a)\inverse = 1\}$.
The map $c_\lambda \colon P_\lambda \to L_\lambda$ given by
\[ c_\lambda(g) = \underset{a\to 0}{\lim}\, \lambda(a)g\lambda(a)\inverse
\]
is a surjective homomorphism of algebraic groups with kernel $R_u(P_\lambda)$;
it coincides with the usual projection $P_\lambda\ra L_\lambda$.  We abuse notation and denote the corresponding map from $P_\lambda^n$ to $L_\lambda^n$ by $c_\lambda$ as well, for any $n\in {\mathbb N}$.
The subgroups $P_\lambda$ for $\lambda \in Y(G)$
are called the \emph{R-parabolic subgroups}
of $G$.
Given an R-parabolic subgroup $P$,
an \emph{R-Levi subgroup} of $P$ is any subgroup $L_\lambda$
such that $\lambda \in Y(G)$ and $P=P_\lambda$.  Note that if $P,Q$ are R-parabolic subgroups of $G$ with $P^0= Q^0$, then $R_u(P)= R_u(Q)$.
If $G$ is connected, then the R-parabolic subgroups
(resp.\ R-Levi subgroups of R-parabolic subgroups)
of $G$ are exactly the parabolic subgroups
(resp.\ Levi subgroups of parabolic subgroups) of $G$;
indeed, most of the theory of parabolic subgroups and
Levi subgroups of connected reductive groups
carries over to R-parabolic and R-Levi subgroups of
arbitrary reductive groups.
In particular, $R_u(P)$ acts simply transitively on the set of
all R-Levi subgroups of an R-parabolic
subgroup $P$.
Also note that $P_\lambda = G$ if and only if $\lambda$ is central in $G$
\cite[Lem.\ 2.4]{BMR}. When it does not cause any confusion, we speak of
``R-Levi subgroups of $G$'' when we mean
``R-Levi subgroups of R-parabolic subgroups of $G$''.
\end{defn}

In this paper, we are interested in reductive subgroups of reductive groups.
If $H$ is a subgroup of $G$, then there is an obvious
inclusion $Y(H) \subseteq Y(G)$ of the sets of cocharacters.
When $H$ is reductive and $\lambda \in Y(H)$, there is then an R-parabolic
subgroup of $H$ associated to $\lambda$, as well as an R-parabolic
subgroup of $G$.
In order to distinguish between R-parabolic subgroups associated to
different subgroups of $G$,
we use the notation
$P_\lambda(H)$, $L_\lambda(H)$, etc., where necessary, but we write
$P_\lambda$ for $P_\lambda(G)$ and $L_\lambda$ for  $L_\lambda(G)$.
Note that $P_\lambda(H) = P_\lambda \cap H$,
$L_\lambda(H) = L_\lambda \cap H$
and $R_u(P_\lambda(H)) = R_u(P_\lambda) \cap H$.

\subsection{Groups acting on varieties}
\label{sub:git}

We recall some general results from geometric invariant theory
required in the sequel, see
\cite{GIT}, \cite[\S 2]{BaRi}, \cite{mumford}, and \cite[Ch.\ 3]{newstead}.

\begin{defn}
\label{defn:limit}
Let $\phi : k^* \to V$ be a morphism of algebraic varieties.
We say that  $\underset{a\to 0}{\lim}\, \phi(a)$ exists
if there exists a morphism $\widehat\phi :k \to V$
(necessarily unique) whose restriction to $k^*$
is $\phi$; if this limit exists, then we set
$\underset{a\to 0}{\lim}\, \phi(a) = \widehat\phi(0)$.
\end{defn}

Now suppose the reductive group $G$ acts on a variety $V$.
For $v \in V$ let $G \cdot v$ denote the $G$-orbit of $v$ in $V$ and
$C_G(v)$ the stabilizer of $v$ in $G$.
It follows easily from Definition \ref{defn:limit}
that if $\underset{a\to 0}{\lim}\,\lambda(a)\cdot v$
exists for a cocharacter $\lambda \in Y(G)$,
then this limit belongs to the closure $\overline{G\cdot v}$
of $G\cdot v$ in $V$.
The well-known Hilbert-Mumford Theorem \cite[Thm.\ 1.4]{kempf}, gives a converse to this:
if $v \in V$ is such that $G\cdot v$ is not closed in $V$,
then there exists a cocharacter $\lambda \in Y(G)$
such that
$\underset{a\to 0}{\lim}\,\lambda(a) \cdot v$ exists and
lies in
the unique closed $G$-orbit contained in $\ovl{G\cdot v}$.

We often use the following simple lemma in the sequel, \cite[Lem.\ 2.12]{GIT}.

\begin{lem}
\label{lem:Ruconj}
Suppose $G$ acts on a variety $V$.
Let $v\in V$, let $\lambda\in Y(G)$ and let
$u\in R_u(P_\lambda)$.
Then $\underset{a \ra 0}{\lim}\lambda(a)\cdot v$
exists and equals $u\cdot v$
if and only if $u^{-1}\cdot\lambda$ centralizes $v$.
\end{lem}


The next result is \cite[Thm.~3.3]{GIT} in case $k = \ovl{k}$.

\begin{thm}
\label{thm:Ruconj}
Suppose $G$ 
acts on a variety $V$.
Let $v\in V$ and let $\lambda\in Y(G)$ such
that $v':=\underset{a\to 0}{\lim} \lambda(a)\cdot v$
exists and is $G$-conjugate to $v$.
Then $v'$ is $R_u(P_\lambda)$-conjugate to $v$.
\end{thm}

\subsection{Generic tuples}
\label{subsec:generictuple}

In order to establish the link between relative $G$-complete
reducibility with respect to $H$ and $H$-orbits of tuples
needed for Theorem \ref{thm:mainthm},
we require the following notion of a generic tuple,
see \cite[Def.~5.4]{GIT}.

\begin{defn}
\label{def:generictuple}
Let $K$ be a subgroup of $G$ and let $G\hookrightarrow\GL_m$
be an embedding of algebraic groups.
Then $\tuple{k} \in K^n$ is
called a \emph{generic tuple of $K$ for the embedding
$G\hookrightarrow\GL_m$} if $\tuple{k}$ generates the
associative subalgebra of $\Mat_m$ spanned by $K$.
We call $\tuple{k}\in K^n$ a \emph{generic tuple of $K$}
if it is a generic tuple of $K$ for some embedding $G\hookrightarrow\GL_m$.
\end{defn}

Clearly, generic tuples exist for any embedding
$G\hookrightarrow\GL_m$ for $n$ sufficiently large.
The main properties of generic tuples are given
by the next lemma, which is \cite[Lem.~5.5]{GIT}.

\begin{lem}
\label{lem:generictuple}
Let $K$ be a subgroup of $G$, let $\tuple{k}\in K^n$
be a generic tuple of $K$ for some embedding $G\hookrightarrow\GL_m$
and let $K'$ be the algebraic subgroup of $G$ generated by
$\tuple{k}$. Then we have:
\begin{enumerate}[{\rm(i)}]
\item $C_H(\tuple{k})=C_H(K')=C_H(K)$ for any subgroup $H$ of $G$;
\item $K'$ is contained in the same R-parabolic
and the same R-Levi subgroups of $G$ as $K$;
\item If $K\subseteq P_\lambda$ for some
$\lambda\in Y(G)$, then $c_\lambda(\tuple{k})$
is a generic tuple of $c_\lambda(K)$ for the
given embedding $G\hookrightarrow\GL_m$.
\end{enumerate}
\end{lem}


\begin{rem}
\label{rem:genericgenerators}
If $K$ is a subgroup of $G$ which is topologically
generated by a tuple ${\tuple k}\in K^n$, then ${\tuple k}$
is a generic tuple of $K$ in the sense of Definition \ref{def:generictuple} \cite[Rem.~5.6]{GIT}.
\end{rem}

\section{Relative $G$-complete reducibility}
\label{sec:relative}
\subsection{Relative $G$-complete reducibility}
\label{subsec:relative}

The key idea for the
proof of Theorem \ref{thm:mainthm}, which is proved in the next
subsection, is the notion of relative $G$-complete reducibility, defined below.

\begin{defn}
\label{defn:relativelystrong}
Let $K$ be a subgroup of $G$ and let $H$ be any reductive subgroup of $G$.
We say that $K$ is
\emph{relatively $G$-completely reducible with respect to $H$} if
for every $\lambda \in Y(H)$ such that $K$ is contained in $P_\lambda$,
there exists $\mu \in Y(H)$ such that
$P_\lambda = P_\mu$ and $K \subseteq L_\mu$.
We sometimes use the shorthand \emph{relatively $G$-cr with respect to $H$}.
\end{defn}

\begin{rems}
\label{rem:relativelystrong}
(i).\ In the case $H=G$, Definition \ref{defn:relativelystrong}
coincides with the usual
definition of $G$-complete reducibility \cite[Sec.~1]{BMR}, \cite[3.2]{serre2};
we sometimes refer to this as the \emph{absolute case}.
Note that a subgroup of $G$ is relatively $G$-completely reducible
with respect to $H$ if and only if it is relatively $G$-completely
reducible with respect to $H^0$.  Thus we may assume without loss that $H$ is connected.  If $H^0$ is central in $G$, then
every subgroup of $G$ is relatively $G$-completely reducible
with respect to $H$.

(ii).\ If $K \subseteq H$, then $K$ is relatively $G$-cr with respect to $H$
if and only if $K$ is $H$-cr (this follows from Lemma \ref{lem:relativelystrong}(ii) below).


(iii).\ In characteristic zero a
subgroup of $G$ is $G$-completely reducible if and only
if it is reductive (cf.\ \cite[Lem.~2.6]{BMR}). We don't know of any simple characterization of
relative $G$-complete reducibility in this case (this complicates the proof of Proposition \ref{prop:SnormalizesK}).
So relative $G$-complete reducibility appears to be a new notion even in characteristic zero.

(iv).\
We note that one
of the basic properties \cite[Prop.~4.1]{serre2} of $G$-cr subgroups
is not inherited by relatively $G$-cr subgroups in general:
if we take $H\subseteq Z(G)$, then all subgroups of $G$ are relatively $G$-cr
with respect to $H$.  In particular, it is possible for
non-reductive (even unipotent) subgroups to be relatively
$G$-cr with respect to
a subgroup $H$.

(v).\
As noted in (iv), in general a relatively $G$-cr subgroup of $G$ need not be
$G$-cr. Also a $G$-cr subgroup need not be relatively $G$-cr.
For instance, let $L$ be an  R-Levi subgroup
of some R-parabolic subgroup $P$ of $G$. Then $L$ is $G$-cr
by \cite[Prop.\ 3.2]{serre2}, \cite[\S 6.3]{BMR}.  Let $M$ be any other R-Levi
subgroup of $P$. Then for any maximal torus $T$ of $G$ that lies in $M$ we
have that $L$ is not relatively $G$-cr with respect to $T$.
For there exists $\lambda \in Y(T)$ with $P = P_\lambda$
and for any such $\lambda$, we have $L_\lambda = M$ by \cite[Cor.~6.5]{BMR}; hence $L\not\subseteq L_\lambda$.  Here is another example: there exists reductive $G$ with a reductive subgroup $H$ and a subgroup $K$ of $H$ such that $K$ is $G$-cr but not $H$-cr \cite[Prop.~7.7]{BMRT}; then $K$ is not relatively $G$-cr with respect to $H$ by (ii) above.
\end{rems}

For examples of relatively $G$-cr subgroups, see
Section \ref{sect:ex} where we specifically study the
case when $G = \GL(V)$.

The following lemma gives some detailed information about conjugacy
of R-Levi subgroups in the subgroups $P_\lambda$ and $P_\lambda(H)$ for $\lambda \in Y(H)$.

\begin{lem}
\label{lem:relativelystrong}
Let $H$ be a reductive subgroup of $G$.
\begin{enumerate}[{\rm (i)}]
\item Let $\lambda,\mu\in Y(H)$ such that $P_\lambda=P_\mu$ and
let $u$ be the element of $R_u(P_\lambda(H))$ such that
$uL_\lambda(H)u\inverse=L_\mu(H)$.
Then $uL_\lambda u\inverse=L_\mu$.
\item Let $K$ be a subgroup of $G$. Then $K$ is relatively
$G$-completely reducible with respect to $H$ if and only if
for every $\lambda \in Y(H)$ such that
$K \subseteq P_\lambda$ there exists $u \in R_u(P_\lambda(H))$
such that $K \subseteq L_{u \cdot \lambda}$.
\end{enumerate}
\end{lem}

\begin{proof}
(i).\
Set $\mu' = u\cdot \lambda$; then $\mu' \in Y(H)$, since $u \in
R_u(P_\lambda(H))$.
We have $L_{\mu'}(H) = uL_\lambda u^{-1}= L_\mu(H)$, so $\mu, \mu' \in Y(Z(L_\mu(H))^0)$.
Let $T$ be a maximal torus of $L_\mu$ containing $Z(L_\mu(H))^0$.
Then since $\mu' \in Y(T)$, we have $T \subseteq L_{\mu'}$,
and $L_{\mu'}$ is an R-Levi subgroup of $P_\lambda$.
Thus $L_\mu$ and $L_{\mu'}$ are R-Levi subgroups of
$P_\lambda$ containing a common maximal torus, so
they are equal, by \cite[Cor.\ 6.5]{BMR}.

(ii).\
This is clear from part (i) and Definition \ref{defn:relativelystrong}.
\end{proof}

\begin{rem}
\label{rem:KvsK'}
Let $H$ be a reductive subgroup of $G$.
Let $K$ be a subgroup of $G$, let $\tuple{k}\in K^n$
be a generic tuple of $K$ and let $K'$ be the
algebraic subgroup of $G$ generated by
$\tuple{k}$. Then it follows from
Lemma \ref{lem:generictuple}(ii) that
$K$ is relatively $G$-cr with respect to $H$ if and only if
$K'$ is.
\end{rem}

\subsection{Relative $G$-complete reducibility and closed orbits}

Throughout the rest of this section, $H$ denotes a reductive subgroup of $G$.
The following result is \cite[Thm.\ 5.8]{GIT} generalized to our
relative setting. In view of Remark \ref{rem:genericgenerators},
the final assertion in part (iii) gives Theorem \ref{thm:mainthm}.
The proof of \cite[Thm.\ 5.8]{GIT} goes through
here with the obvious modifications and alterations of notation, replacing \cite[Thm.~3.3]{GIT} with Theorem \ref{thm:Ruconj} and \cite[Lem.~5.5]{GIT} with Lemma \ref{lem:generictuple}.

\begin{thm}\ 
\label{thm:orbitcrit}
Let $H$ be a
reductive subgroup of $G$.
\begin{enumerate}[{\rm (i)}]
\item Let $n\in {\mathbb N}$,
let $\tuple{k}\in G^n$ and let $\lambda\in Y(H)$
such that $\tuple{m}:=\lim_{a\ra 0}\lambda(a)\cdot \tuple{k}$ exists.
Then the following are equivalent:
\begin{enumerate}[{\rm (a)}]
\item $\tuple{m}$ is $H$-conjugate to $\tuple{k}$;
\item $\tuple{m}$ is $R_u(P_\lambda(H))$-conjugate to $\tuple{k}$;
\item $\dim H\cdot\tuple{m}=\dim H\cdot\tuple{k}$.
\end{enumerate}
\item Let $K$ be a subgroup of $G$ and let $\lambda\in Y(H)$.
Suppose $K\subseteq P_\lambda$ and set $M= c_\lambda(K)$.
Then $\dim C_H(M)\geq \dim C_H(K)$ and the following are equivalent:
\begin{enumerate}[{\rm (a)}]
\item $M$ is $H$-conjugate to $K$;
\item $M$ is $R_u(P_\lambda(H))$-conjugate to $K$;
\item $K \subseteq L_\mu$ for some $\mu\in Y(H)$
such that $P_\lambda = P_\mu$;
\item $\dim C_H(M)=\dim C_H(K)$.
\end{enumerate}
\item Let $K$, $\lambda$ and $M$ be as in (ii) and let
$\tuple{k}\in K^n$ be a generic tuple of $K$.
Then the assertions in (i) are equivalent to those in (ii).
    In particular, $K$ is relatively $G$-completely reducible
with respect to $H$ if and only if $H\cdot\tuple{k}$ is closed in $G^n$.
\end{enumerate}
\end{thm}

\begin{cor}
\label{cor:GvsM}
Let $M$ be a reductive subgroup of $G$.
Let $K$ and $H$ be subgroups of $M$
and assume that $H$ is reductive.
Then $K$ is relatively $G$-completely reducible with respect to $H$
if and only if it is
relatively $M$-completely reducible with respect to $H$.
\end{cor}

\begin{proof}
Clearly, if $\tuple{k}$ is a generic tuple of $K$ with respect to an embedding of $G$ in ${\rm GL}_m$, then it is a generic tuple of $K$ with respect to the embedding of $M$ in ${\rm GL}_m$ obtained by restriction.  Since a subset of $M^n$ is closed
if and only if it is closed in $G^n$,
the result follows immediately from Theorem \ref{thm:orbitcrit}(iii).
\end{proof}

We have an analogue of part of \cite[Prop.~3.12]{BMR}:

\begin{cor}
\label{cor:bmr3.12}
Let $H$ be a reductive subgroup of $G$, and let $K$ be a subgroup of $G$ which
is relatively $G$-completely reducible with respect to  $H$.
Then $C_H(K)$ is reductive.
\end{cor}

\begin{proof}
Let $\tuple k$ be a generic tuple of $K$.
Since $K$ is relatively $G$-cr with respect to  $H$, the orbit
$H\cdot \tuple k$ is closed
in $G^n$, by Theorem \ref{thm:orbitcrit}(iii), and therefore affine.
Hence $C_H(\tuple k) = C_H(K)$ is reductive, by
\cite[Lem.~10.1.3]{rich1}.
\end{proof}

\subsection{Relative $G$-complete reducibility for Lie subalgebras of $\mathfrak{g}$}

It is straightforward to extend our definitions and results to
Lie subalgebras of $\mathfrak{g}$.
We first record a standard result which gives some properties
of the Lie algebras of R-parabolic and R-Levi subgroups of $G$
(cf.~\cite[\S 2.1]{rich}).

\begin{lem}\label{lem:liealgebrasofRpars}
For $\lambda\in Y(G)$, put $\pp_\lambda=\Lie(P_\lambda)$ and
$\frakl_\lambda=\Lie(L_\lambda)$.
Let $x\in\gg$. Then
\begin{enumerate}[{\rm(i)}]
\item $x\in\pp_\lambda$ if and only if $\underset{a\to 0}{\lim}\,
\lambda(a)\cdot x$ exists;
\item $x\in\frakl_\lambda$ if and only if $\underset{a\to 0}{\lim}\,
\lambda(a)\cdot x$ exists and equals $x$ if and only if $\lambda(k^*)$ centralizes $x$;
\item $x\in\Lie(R_u(P_\lambda))$ if and only if $\underset{a\to 0}{\lim}\,
\lambda(a)\cdot x$ exists and equals $0$.
\end{enumerate}
\end{lem}

\begin{defn}
\label{def:relLieGcr}
For $\lambda\in Y(G)$ define the subalgebras $\pp_\lambda$ and $\frakl_\lambda$
of $\gg$ as in Lemma~\ref{lem:liealgebrasofRpars}.
Let $\kk$ be a subalgebra of $\gg$ and let $H$ be any reductive subgroup of $G$.
We call $\kk$
\emph{relatively $G$-completely reducible with respect to $H$} if
for every $\lambda \in Y(H)$ such that $\kk\subseteq\pp_\lambda$,
there exists $\mu \in Y(H)$ such that
$P_\lambda = P_\mu$ and $\kk \subseteq \frakl_\mu$.
In case $H = G$, we say
$\kk$ is \emph{$G$-completely reducible}.
\end{defn}


The following can be shown with the same arguments as
Theorem~\ref{thm:orbitcrit}. We emphasize that
Theorem~\ref{thm:relLieGcr}(iii)
characterizes the closed $H$-orbits in $\gg^n$.  We define the map $c_\lambda\colon \pp_\lambda\ra \frakl_\lambda$ for Lie algebras in the obvious way (cf.\ Section \ref{subsec:noncon}).

\begin{thm}
\label{thm:relLieGcr}
Let $H$ be a reductive subgroup of $G$.
\begin{enumerate}[{\rm (i)}]
\item Let $n\in {\mathbb N}$,
let $\tuple{k}\in \gg^n$ and
let $\lambda\in Y(H)$ such that
$\tuple{m}:=\lim_{a\ra 0}\lambda(a)\cdot \tuple{k}$ exists.
Then the following are equivalent:
\begin{enumerate}[{\rm (a)}]
\item $\tuple{m}$ is $H$-conjugate to $\tuple{k}$;
\item $\tuple{m}$ is $R_u(P_\lambda(H))$-conjugate to $\tuple{k}$;
\item $\dim H\cdot\tuple{m}=\dim H\cdot\tuple{k}$.
\end{enumerate}
\item Let $\kk$ be a Lie subalgebra of $\gg$ and let $\lambda\in Y(H)$.
Suppose $\kk\subseteq \pp_\lambda$ and set $\mm= c_\lambda(\kk)$.
Then $\dim C_H(\mm)\geq \dim C_H(\kk)$ and the following are equivalent:
\begin{enumerate}[{\rm (a)}]
\item $\mm$ is $H$-conjugate to $\kk$;
\item  $\mm$ is $R_u(P_\lambda(H))$-conjugate to $\kk$;
\item $\kk \subseteq \frakl_\mu$ for some $\mu\in Y(H)$
such that $P_\lambda = P_\mu$;
\item $\dim C_H(\mm)=\dim C_H(\kk)$.
\end{enumerate}
\item Let $\kk$, $\lambda$ and $\mm$ be as in (ii),
and suppose $\tuple{k}\in \kk^n$ is a generating tuple for $\kk$ for some $n\in {\mathbb N}$.
Then the assertions in (i)
are equivalent to those in (ii).
    In particular, $\kk$ is relatively $G$-completely reducible
with respect to $H$ if and only if $H\cdot\tuple{k}$ is closed in $\gg^n$.
\end{enumerate}
\end{thm}

\begin{rem}
\label{rems:relLieGcr}
For $H = G$, Definition \ref{def:relLieGcr}
is due to G.\ McNinch, \cite{mcninch}.
The final statement of Theorem \ref{thm:relLieGcr}(iii)
generalizes \cite[Thm.\ 1(1)]{mcninch}
which is the case $H = G$.
Note also that Theorem \ref{thm:relLieGcr} generalizes \cite[Thm.\ 5.26]{GIT}
which is the case $H = G$.
\end{rem}

\subsection{Relative $G$-complete reducibility for associative subalgebras
of $\End(V)$}

Let $G=\GL(V)$ and let $H$ be any reductive subgroup of $G$.
Using the characterization of parabolic and Levi
subgroups of $G$ in terms of flags in $V$,
we see that $\pp_\lambda$ and $\frakl_\lambda$, for
$\lambda\in Y(H)$, are associative subalgebras of $\End(V)$. This
means that one can define the notion of
\emph{relative $G$-complete reducibility with
respect to $H$} for associative subalgebras of $\End(V)$ in the obvious way.
Observe that in the absolute case, i.e., when $H = \GL(V)$, we obtain that
an associative subalgebra $A$ of $\End(V)$ is $G$-completely reducible
precisely when $V$ is a semisimple $A$-module.

Obviously, a subgroup of $G$ or a Lie subalgebra of $\gg=\gl(V)$ is
relatively $G$-cr with respect to $H$ if and only the same
holds for the associative subalgebra of $\End(V)$ that it
generates.
If $K$ is a subgroup of $G$ which is topologically
generated by a tuple ${\bf k}\in K^n$, then ${\bf k}$ generates the
associative subalgebra of $\End(V)$ generated by $K$
(cf.\ Remark \ref{rem:genericgenerators}).
Similarly, if $\kk$ is a Lie subalgebra of $\gg$ which is generated by
a tuple ${\bf k}\in\kk^n$, then ${\bf k}$ generates the associative
subalgebra of $\End(V)$ generated by $\kk$.

If a tuple ${\bf a}\in\End(V)^n$ generates the
associative subalgebra $A$ of $\End(V)$, then the analogue
of Theorem~\ref{thm:orbitcrit} holds; in particular, $A$ is
relatively $G$-cr with respect to $H$ if and only if $H\cdot{\bf a}$
is closed.
Therefore, in the absolute case $H = G = \GL(V)$, the final statement of
the analogue of Theorem~\ref{thm:orbitcrit}(iii)
in this setting recovers a fundamental result
in the representation theory of associative algebras
due to H.~Kraft:
$V$ is a semisimple $A$-module if and only if
the $G$-orbit $G \cdot{\bf a}$ is closed in $\End(V)^n$,
see \cite[Prop.\ 4.4]{kraft1} or \cite[II.2.7 Satz 2]{kraft}.
In that sense, this concept of
relative $\GL(V)$-complete reducibility with respect to $H$ for
associative subalgebras of $\End(V)$
generalizes this work of Kraft.

\subsection{$H$-stable points in $G^n$ and relative $G$-irreducibility with
respect to $H$}

Recall the notion of a \emph{stable point} for the action of a reductive group
$G$ on a variety $V$ \cite[1.4]{rich1}:

\begin{defn}
\label{defn:stability}
Let $Z = \bigcap_{v \in V}C_G(v)$ be the kernel of the action
of $G$ on $V$.
We say that $v \in V$ is a
\emph{stable point} for the action of $G$ or a
\emph{$G$-stable point} provided the orbit $G \cdot v$ is closed in $V$ and
$C_G(v)/Z$ is finite.
\end{defn}

\begin{rem}
\label{rem:stability}
Let $V/\!\!/G$ be the variety corresponding to the $k$-algebra $k[V]^G$ and let
$\pi : V \to V/\!\!/G$ be the morphism corresponding to the inclusion
$k[V]^G \to k[V]$. In general, $\pi$ is not a quotient morphism
in the sense of \cite[\S 6]{Bo} because $\pi$ might have a fibre which is the union of more than one $G$-orbit.
Let $V^s$ denote the set of $G$-stable points in $V$.
Then $V^s$ is a (possibly empty)
$G$-stable open subset of $V$ and $V^s = \pi^{-1}(\pi(V^s))$.
Further, $\pi(V^s)$ is an open subset of $V/\!\!/G$ and is a
geometric quotient of $G$.
For $v \in V^s$, we have $G \cdot v = \pi^{-1}(\pi(v))$
(see \cite[Sec.\ 4]{martin1}).
\end{rem}

In \cite[Prop.\ 16.7]{rich}, Richardson characterizes the $G$-stable points in $G^n$.
We can easily extend this result to the $H$-stable points in $G^n$ for reductive subgroups
$H$ of $G$.
To do this, we first extend the notion of $G$-irreducibility from \cite[\S 3.2]{serre2}
to the relative setting:

\begin{defn}
\label{def:GirrGind}
Let $H$ and $K$ be subgroups of $G$ with $H$ reductive.
We say that
\emph{$K$ is relatively $G$-irreducible ($G$-ir) with respect to $H$}
if whenever $\lambda \in Y(H)$ and $K \subseteq P_\lambda$, we have $P_\lambda = G$.
For $H = G$, this relative notion agrees with that of
$G$-irreducibility, cf.\ \cite[\S 3.2]{serre2}, \cite[\S 2.4]{BMR}.
\end{defn}

\begin{rem}
Obviously, relative $G$-irreducibility with respect to $H$ implies
relative $G$-complete reducibility with respect to $H$.
It is clear that for subgroups $H$, $K$ and $M$ of $G$ with
$H$ and $M$ reductive and $H, K\subseteq M$,
if $K$ is relatively $G$-irreducible with respect to $H$, then
$K$ is relatively $M$-irreducible with respect to $H$,
cf.\ \cite[Cor.\ 2.7]{BMR}.
\end{rem}

The notion of relative $G$-irreducibility is exactly what we need to characterize
$H$-stability in $G^n$.
In view of Remark \ref{rem:genericgenerators},
our next result generalizes \cite[Prop.\ 16.7]{rich}; 
see also \cite[Prop.\ 2.13]{BMR}.  
Since Richardson's proof applies mutatis mutandis,
we do not include it.

\begin{prop}
\label{prop:rich-stable}
Let $H$ be a reductive subgroup of $G$.
Let $K$ be a subgroup of $G$ and let $\tuple k \in K^n$
be a generic tuple of $K$.
Then $K$ is relatively $G$-irreducible with respect to $H$
if and only if
$\tuple k$ is an  $H$-stable point in $G^n$.
\end{prop}

The following provides analogues of Corollaries~3.22 and 3.5 of \cite{BMR}.

\begin{prop}
\label{prop:relGcrvsGir}
Let $H$ and $K$ be subgroups of $G$ and suppose that $H$ is reductive.
\begin{enumerate}[{\rm(i)}]
\item Let $S$ be a torus of $C_H(K)$ and set $L = C_G(S)$.
Then $K$ is relatively $G$-completely reducible with respect to $H$
if and only if $K$ is relatively $L$-completely reducible with
respect to $H \cap L$.
\item The R-Levi subgroups $L_\mu$ of $G$ for $\mu \in Y(H)$
that are minimal with respect to containing
$K$ are precisely the subgroups of the form $C_G(S)$ where $S$ is a maximal torus
of $C_H(K)$.
If $L$ is such an R-Levi subgroup of $G$, then $K$ is relatively
$G$-completely reducible with respect to $H$
if and only if $K$ is relatively $L$-irreducible with respect to $H\cap L$.
\end{enumerate}
\end{prop}

\begin{proof}
First we prove the first assertion of (ii).
If $S$ is a torus in $G$, then $C_G(S)=L_\lambda$
for some $\lambda\in Y(S)$ by the arguments of the
proof of \cite[Cor.~6.10]{BMR}. Now assume that $S$
is a maximal torus of $C_H(K)$. If $\mu\in Y(H)$ such
that $K\subseteq L_\mu\subseteq C_G(S)$, then $\mu(k^*)$
commutes with $S$ and is contained in $C_H(K)$. So, by
the maximality of $S$, $\mu(k^*)\subseteq S$ and $L_\mu=C_G(S)$.

Now assume that $\lambda\in Y(H)$ and $L_\lambda$ is minimal among the R-Levi
subgroups $L_\mu$ of $G$ with $\mu \in Y(H)$ and $K\subseteq L_\mu$.
Put $S=(Z(L_\lambda)\cap H)^0$. Then one easily checks that
$L_\lambda=C_G(S)$. Let $T$ be a torus of $C_H(K)$
with $S\subseteq T$. Then $K\subseteq C_G(T)\subseteq C_G(S)=L_\lambda$.
As we have seen above, $C_G(T)=L_\mu$ for some $\mu\in Y(T)\subseteq Y(H)$.
So, by the minimality of $L_\lambda$, $C_G(T)=L_\lambda$ and $T\subseteq S$.
So $S$ is a maximal torus of $C_H(K)$.

To prove (i) and the second assertion of (ii) let $\tuple k$ be a
generic tuple of $K$.
Let $S$ be a torus of $C_H(K)$ and let $L = C_G(S)$.  Note that $H\cap L= C_H(S)$ is reductive.
By Theorem \ref{thm:orbitcrit}(iii),
$K$ is relatively $G$-completely reducible with respect to $H$
if and only if
$H\cdot \tuple k$ is closed in $G^n$ and likewise if we replace
$G$ and $H$ by $L$ and $H \cap L$, respectively.
So (i) now follows from \cite[Thm.~C]{rich1}.

Now assume that $S$ is a maximal torus of $C_H(K)$.
Then $H\cdot \tuple k$ is closed in $G^n$
if and only if
$\tuple k$ is a stable point for the $(H\cap L)$-action
on $L^n$ by \cite[Lem.~16.6]{rich}.
By Proposition \ref{prop:rich-stable}, this is equivalent to
$K$ being relatively $L$-irreducible with respect to $H\cap L$.
\end{proof}

\begin{lem}
\label{lem:paroflevi}
Let $H$ be a reductive subgroup of $G$ and let $\lambda\in Y(H)$.
Then the R-parabolic subgroups $P_\mu$  of $G$  contained in $P_\lambda$
with $\mu \in Y(H)$ are precisely the subgroups of the form
$Q \ltimes R_u(P_\lambda)$, where $Q = P_\nu(L_\lambda)$
and $\nu \in Y(L_\lambda(H))$.
\end{lem}

\begin{proof}
If $\mu\in Y(H)$ such that $P_\mu\subseteq P_\lambda$,
then $P_\mu$ is of the stated form if some $R_u(P_\lambda(H))$-conjugate of $P_\mu$
is of this form. Now we can replace $\mu$ with an $R_u(P_\lambda(H))$-conjugate
which lies in $Y(L_\lambda(H))$. The rest of the proof is completely analogous
to that of \cite[Lem.\ 6.2(ii)]{BMR}.
\end{proof}

We can now generalize \cite[Prop.\ 5.14]{GIT},
showing that we can associate to each $H$-conjugacy
class of subgroups of $G$ a unique
$H$-conjugacy class of subgroups which are relatively
$G$-cr with respect to $H$; Theorem~\ref{thm:uniquegcr}
below is the
group-theoretic analogue of the statement that the closure of
each $H$-orbit in $G^n$ contains a unique closed $H$-orbit.

\begin{thm}
\label{thm:uniquegcr}
Let $H$ and $K$ be subgroups of $G$ and suppose that $H$ is reductive.
\begin{enumerate}[{\rm(i)}]
\item There exists $\lambda\in Y(H)$ and a subgroup $M$
of $G$ which is relatively $G$-completely reducible
with respect to $H$ such that $K\subseteq P_\lambda$ and $c_\lambda(K)=M$.
Moreover, $M$ is unique up to $H$-conjugacy and its $H$-conjugacy class
depends only on the $H$-conjugacy class of $K$.
\item Any automorphism of the algebraic group $G$ that normalizes $H$
and stabilizes the $H$-conjugacy class of $K$, stabilizes the $H$-conjugacy
class of $M$.
\item If $\mu \in Y(H)$ and $K\subseteq P_\mu$,
then the procedure described in (i) associates the same $H$-conjugacy class of subgroups to $K$ and $c_\mu(K)$.
\end{enumerate}
\end{thm}

\begin{proof}
Let $\lambda \in Y(H)$ be such that $P_\lambda$ is minimal among
the R-parabolic subgroups $P_\mu$ of $G$ with $\mu \in Y(H)$ and
$K \subseteq P_\mu$. Then $P_\lambda$ is also minimal among the
R-parabolic subgroups $P_\mu$ of $G$ with $\mu \in Y(H)$ and
$c_\lambda(K)\subseteq P_\mu$ by the same arguments as in the
proof of \cite[Prop.~5.14]{GIT}. It now follows from
Lemma~\ref{lem:paroflevi} that $c_\lambda(K)$ is relatively
$L_\lambda$-irreducible with respect to $L_\lambda(H)$.
But then $c_\lambda(K)$ is relatively $G$-cr with respect
to $H$ by Proposition~\ref{prop:relGcrvsGir}(i).

The rest of the proof is completely analogous to
that of \cite[Prop.\ 5.14]{GIT}. One has to conjugate
with elements from $R_u(P_\lambda(H))$ rather than
$R_u(P_\lambda)$. The cocharacters $\lambda$ and $\mu$
in the proof of \cite[Prop.\ 5.14]{GIT} can now be put
in a common maximal torus of $P_\lambda(H)$ and $P_\mu(H)$.
\end{proof}

\begin{rem}
\label{rem:uniqueLiegcr}
A statement analogous to Theorem \ref{thm:uniquegcr}
holds for Lie algebras: that is, given any Lie subalgebra $\kk$ of $\gg$,
we can find a uniquely defined $H$-conjugacy class of subalgebras of $\gg$
which contains $c_\lambda(\kk)$ for some  $\lambda \in Y(H)$, each member of which
is relatively $G$-cr with respect to $H$.
\end{rem}

\subsection{Optimal parabolic subgroups}
\label{sec:optpars}

Let $K$ be a subgroup of $G$.
If $K$ is not relatively $G$-cr with respect to $H$, then there exists at least one cocharacter $\lambda \in Y(H)$
such that $K \subseteq P_\lambda$, but $K \not\subseteq L_{u\cdot\lambda}$ for any $u \in R_u(P_\lambda(H))$.
Following work in \cite[Sec.~5]{GIT}, we now show how to make a so-called optimal choice for this $\lambda \in Y(H)$;
being able to make such a choice has many advantages and shortens some of the proofs which follow.
Since the constructions we are going to discuss are very similar to those in \cite[Sec.~5]{GIT}, where the case of
a non-$G$-cr subgroup of $G$ is addressed, we omit some details and content ourselves with
pointing out the necessary modifications to allow the results to go through here.  The constructions below in the special case $H=G$ agree with those of {\em loc.\ cit.}

We first need to adapt some of the notation from \cite[Sec.~4, Sec.~5]{GIT} to the relative setting.
Suppose $K$ is a subgroup of $G$ and suppose $\lambda \in Y(H)$ is such that $K \subseteq P_\lambda$.
Set $M := c_\lambda(K)$ and let $S_n(M) = \overline{H\cdot M^n}$.
Then $K^n$ is a uniformly $S_n(M)$-unstable set for the action
of $H$ on $G^n$ in the sense of \cite[Sec.~4]{GIT}.
Any $G$-invariant norm on $Y(G)$ in the sense of \cite[Def.~4.1]{GIT} restricts
to an $N_G(H)$-invariant norm on $Y(H)$;
let $\left\|\,\right\|$ be such a norm.
Then \cite[Sec.~4]{GIT} provides a set
$\Omega(K^n,S_n(M))$ of cocharacters of $H$, the so-called \emph{optimal class} (note that this depends on $H$, $n$ and $\left\|\,\right\|$ as well as $M$, but we suppress these other dependences in this and subsequent notation; cf.\ \cite[Rem.~5.22]{GIT}).
Similarly, if $\kk$ is a Lie subalgebra of $\gg$ and $\lambda \in Y(H)$ is such that $\kk \subseteq \pp_\lambda$,
then setting $\mm := c_\lambda(\kk)$ and $S_n(\mm) = \overline{H\cdot \mm^n}$, we get an optimal class
$\Omega(\kk^n,S_n(\mm)) \subseteq Y(H)$.

We have the following analogue of \cite[Thm.\ 5.16]{GIT} in the relative setting:

\begin{thm}
\label{thm:optpar}
Let $K$ be a subgroup of $G$ and let $n\in {\mathbb N}$ such that
$K^n$ contains a generic tuple of $K$.
Let $M$ be a subgroup of $G$ and suppose that $M= c_\lambda(K)$ for
some $\lambda\in Y(H)$ with $K\subseteq P_\lambda$.
Put $\Omega(K,M) := \Omega(K^n,S_n(M))$. Then the following hold:
\begin{enumerate}[{\rm (i)}]
\item $P_\mu = P_\nu$ for all $\mu, \nu \in \Omega(K,M)$.
Let $P(K,M)$ denote the unique R-parabolic subgroup of $G$
so defined. Then $K\subseteq P(K,M)$ and $R_u(P(K,M)\cap H)$
acts simply transitively on $\Omega(K,M)$.
\item For $g\in N_G(H)$ we have $\Omega(gKg^{-1},gMg^{-1})=
g\cdot\Omega(K,M)$ and $P(gKg^{-1},gMg^{-1})=  gP(K,M)g^{-1}$.
If $g\in G$ normalizes $H$, $K$ and $S_n(M)$, then $g\in P(K,M)$.
\item If $\mu\in\Omega(K,M)$, then $\dim C_H(c_\mu(K))\ge\dim C_H(M)$.
If $M$ is $H$-conjugate to $K$, then $\Omega(K,M)=\{0\}$ and $P(K,M) = G$.
If $M$ is not $H$-conjugate to $K$, then $K$ is not contained in $L_\mu$
for any $\mu \in Y(H)$ with $P_\mu = P(K,M)$.
\end{enumerate}
\end{thm}

\begin{proof}
We apply \cite[Thm.~4.5]{GIT} with $(G',G,V,X,S)=(N_G(H),H,G^n,K^n,S_n(M))$.
Since we associate to $K$ and $M$ an R-parabolic subgroup of $G$ rather than
$H$ in {\em loc.~cit.}, we need to give some more arguments.

(i) and (ii). The statements about $ \Omega(K,M)$
follow immediately from \cite[Thm.~4.5]{GIT}.
Let $\mu \in\Omega(K,M)$. By \cite[Thm.~4.5(iv)]{GIT},
$R_u(P_\mu(H)) = R_u(P_\mu) \cap H$ acts simply transitively
on $\Omega(K,M)$. Hence $P_\mu =P_\nu$ for all $\mu,\nu \in \Omega(K,M)$.
The final assertion in (ii) is proved in the same way as the final assertion of \cite[Thm.~4.5(iv)]{GIT}.

(iii).
The proof of this  is completely analogous to that of
\cite[Thm.~5.16(iii)]{GIT}.
Note that $P(K,M) = G$
implies that $\Omega(K,M)$ consists of the trivial cocharacter of $H$ only.
\end{proof}

\begin{rem}
Note that Theorem~\ref{thm:uniquegcr} provides an obvious choice for
the subgroup $M$ in Theorem \ref{thm:optpar}:
for by Theorem~\ref{thm:uniquegcr}, if we are given a subgroup $K$, then there is a
unique conjugacy class of subgroups of the form
$M = c_\lambda(K)$ for $\lambda \in Y(H)$ which
are relatively $G$-cr with respect to $H$.
Since these subgroups are all $H$-conjugate, the
set $S_n(M)$ does not depend on which representative $M$ we choose
from this conjugacy class, and hence the optimal
destabilizing R-parabolic subgroup $P(K,M)$ of $G$
also does not depend on the choice
of $M$ from this class. This leads to the following definition.
\end{rem}

\begin{defn}
\label{defn:optpar}
Let $K$ be any subgroup of $G$,
and let $M$ be a representative from the $H$-conjugacy class
attached to $K$ of subgroups which are relatively $G$-cr with
respect to $H$, provided by Theorem~\ref{thm:uniquegcr}.
Define $\Omega(K) = \Omega(K,M)$ and $P(K) = P(K,M)$.
By Theorems \ref{thm:uniquegcr} and \ref{thm:optpar},
$K$ and $N_{N_G(H)}(K)$ are contained in $P(K)$ and
for $\mu\in\Omega(K)$, $c_\mu(K)$ is relatively
$G$-completely reducible with respect to $H$.
So, by Theorem~\ref{thm:orbitcrit}(ii), if $K$ is not relatively
$G$-completely reducible with respect to $H$, then $K$ is not
contained in any R-Levi subgroup of $P(K)$ of the form $L_\mu$
with $\mu \in Y(H)$ and $P_\mu = P(K)$.
Note that, trivially, $P(K)=G$ if $K$ is relatively
$G$-completely reducible with respect to $H$.
We call $\Omega(K)$ the \emph{optimal class for $K$ relative to $H$}
and $P(K)$ the
\emph{optimal destabilizing R-parabolic subgroup for $K$ relative to $H$}.
In case $H = G$, we call $P(K)$ the
\emph{optimal destabilizing R-parabolic subgroup for $K$}.
\end{defn}

\begin{rem}
One can give analogues of Theorem \ref{thm:optpar} and Definition \ref{defn:optpar} for a Lie subalgebra
$\kk$ of $\gg$;
note that one can work with generating tuples for $\kk$ in place of generic tuples.
We leave the details to the interested reader.
\end{rem}

Armed with Theorems \ref{thm:mainthm}, \ref{thm:orbitcrit}
and \ref{thm:optpar},
we can extend many results about $G$-complete reducibility
within the framework of relative $G$-complete reducibility.
In the following subsections we aim to illustrate interesting
points of our new construction
by looking at a selection of results, mainly from \cite{BMR};
some of these results generalize immediately, while others are more subtle.
We observe that all the results below have interpretations in terms of
closedness of orbits in $G^n$ in view of Theorem \ref{thm:mainthm}.

\subsection{New relatively $G$-completely reducible subgroups from old}
\label{sub:newfromold}

In this subsection we explore how
to generate new relatively $G$-cr subgroups from a given relatively $G$-cr subgroup $M$;
for example, by taking suitable normal subgroups of $M$ or looking at suitable subgroups of $N_G(M)$.
Our first result generalizes \cite[Thm.~3.10]{BMR},
as the latter is simply the special case $H = G$
of Theorem \ref{thm:relativenormal}.
The apparent direct analogue of \cite[Thm.~3.10]{BMR} in this context,
namely that a normal subgroup $N$ of a subgroup $M$ of
$G$ is relatively $G$-cr with
respect to some reductive subgroup $H$ of $G$ provided $M$ is,
fails in general; see Examples \ref{exmp:relative1} and \ref{exmp:relative2} below.

\begin{thm}
\label{thm:relativenormal}
Let $H$ be a reductive subgroup of $G$ and
let $K$ and $M$ be subgroups of $G$ such that
$K\subseteq M\subseteq KN_{N_G(H)}(K)$. If $M$ is
relatively $G$-completely reducible with respect to $H$, then so is $K$.
\end{thm}

\begin{proof}
We prove the contrapositive.
So suppose $K$ is not relatively $G$-cr with respect to $H$.
Let $P(K)$ be the optimal R-parabolic destabilizing subgroup for $K$ with respect to $H$ (Definition \ref{defn:optpar}).
Then by Definition~\ref{defn:optpar}, $M\subseteq KN_{N_G(H)}(K)\subseteq P(K)$ and
$K$ is not contained in any R-Levi subgroup of $P(K)$ of the form
$L_\mu$ with $\mu \in Y(H)$ and $P_\mu = P(K)$.
Hence $M$ cannot be contained in such an R-Levi subgroup of $P(K)$,
as $K \subseteq M$. Thus $M$ is not relatively $G$-cr with respect to $H$.
\end{proof}

The following result generalizes the
second statement of \cite[Prop.\ 3.19]{BMR}.

\begin{prop}
\label{prop:bmr3.19}
Suppose $H$ is a reductive subgroup of $G$.
Let $K$ be a subgroup
of $G$ which is relatively $G$-completely reducible with respect to $H$,
and suppose $M$ is a reductive subgroup of $G$ which contains $K$ and
is normalized by a maximal torus of $C_H(K)$.
Then $M$ is also relatively $G$-completely reducible with respect to $H$.
\end{prop}

\begin{proof}
Choose a maximal torus $S$ of $C_H(K)$ which normalizes $M$.
Since $MS$ is reductive and $M\subseteq MS\subseteq MN_H(M)$,
we may assume that $S\subseteq M$, by
Theorem~\ref{thm:relativenormal}.
Let $\lambda\in Y(H)$ be such that $M \subseteq P_\lambda$.
Since $M$ contains $K$ and $K$ is relatively $G$-cr with
respect to $H$, we may assume that $K\subseteq L_\lambda$.
Then we have $\lambda(k^*)\subseteq C_H(K)\cap P_\lambda$.
But $S$ is a maximal torus of $C_H(K)\cap P_\lambda$, so there exists
$g\in C_H(K)\cap P_\lambda$ such that for $\mu:=g\cdot\lambda$
we have $\mu(k^*)\subseteq S\subseteq M$. Clearly, we also have
$\mu\in Y(H)$ and $P_\mu=P_\lambda$.
Since $M \subseteq P_\lambda = P_\mu$, we have $P_\mu(M)=M$.
Since $M$ is reductive, this means that $\mu(k^*)$ is central in $M$,
by Definition \ref{defn:rpars}. So $M\subseteq L_\mu$, as required.
\end{proof}

\begin{rem}
Note that Proposition \ref{prop:bmr3.19} applies in particular to reductive
subgroups of $G$ which contain all of $KC_H(K)$.
In particular, if $K$ is relatively $G$-completely reducible with respect to $H$,
then, provided they are reductive, $N_G(K)$ and $KC_G(K)$ are
relatively $G$-completely reducible with respect to $H$.
This generalizes \cite[Cor.~3.16]{BMR}; recall that
$N_G(K)$ and $KC_G(K)$ are automatically reductive
if $K$ is a $G$-cr subgroup of $G$, cf.~\cite[Prop.~3.12]{BMR}.
\end{rem}

The following result generalizes \cite[Prop.~3.20]{BMR}.

\begin{cor}
\label{cor:bmr3.19}
Let $H$ be a reductive subgroup of $G$.
Then any reductive subgroup of $G$ which is normalized by
a maximal torus of $H$ is relatively $G$-completely
reducible with respect to $H$.
\end{cor}

\begin{proof}
Let $K$ be a reductive subgroup of $G$ which is normalized by a
maximal torus of $H$.
By Theorem \ref{thm:relativenormal}, we may assume that $K$
contains a maximal torus $S$ of $H$.
Now $S$ is $H$-cr \cite[Lem.~2.6]{BMR}, so $S$ is relatively $G$-cr with respect to $H$ (Remark \ref{rem:relativelystrong}(ii)).  The result follows from Proposition \ref{prop:bmr3.19} applied to the inclusion $S\subseteq K$.
\end{proof}


\subsection{Relative $G$-complete reducibility with respect to different subgroups of $G$}
\label{sub:varyH}
In our next set of results we explore
what happens when we vary the reductive subgroup
$H$, rather than  $K$.
The first result generalizes one direction of \cite[Prop.~2.8]{BMR2}.
Note that the converse in this case is not true;
just take any example of $G$, $H$
and $K$ where $K$ is not relatively $G$-cr with
respect to $H$, and let $N = \{1\}$. 

\begin{prop}\label{prop:normalrelative}
Suppose $H$ is a reductive subgroup of $G$ and $N$ is a normal subgroup of $H$.
For any subgroup $K$ of $G$,
if $K$ is relatively $G$-completely reducible with respect to $H$,
then $K$ is relatively $G$-completely reducible with respect to $N$.
\end{prop}

\begin{proof}
First note that since $H$ is reductive and $N$ is normal in $H$, $N$ is reductive.
Suppose $\lambda \in Y(N) \subseteq Y(H)$ is such that $K \subseteq P_\lambda$.
Then, as $K$ is relatively $G$-cr with respect to $H$, there
exists $u \in R_u(P_\lambda(H))$ such that $K \subseteq uL_\lambda u\inverse = L_{u\cdot\lambda}$,
by Lemma \ref{lem:relativelystrong}.
But $u \in H$, so $u$ normalizes $N$, and thus $u\cdot\lambda \in Y(N)$,
and we are done.
\end{proof}

\begin{cor}
Let $H$ and $K$ be subgroups of $G$ such that $H$ is reductive and $K$ is relatively $G$-completely reducible
with respect to $H$.
Suppose $M$ is a reductive subgroup of $H$ which is normalized by $C_H(K)$.
Then $K$ is relatively $G$-completely reducible with respect to $M$.
\end{cor}

\begin{proof}
Since $C_H(K)$ normalizes $M$, $M$ is normal in $MC_H(K)$.
By Proposition \ref{prop:normalrelative},
if $K$ is relatively $G$-cr with respect to $MC_H(K)$, then $K$ is relatively
$G$-cr with respect to $M$, so we may assume that $C_H(K) \subseteq M$.
Let $\lambda \in Y(M)$ such that $K \subseteq P_\lambda$.
Then, as $K$ is relatively $G$-cr with respect to $H$, there
exists $u \in R_u(P_\lambda(H))$ such that $K \subseteq L_{u\cdot\lambda}$,
by Lemma \ref{lem:relativelystrong}.
But then $u\cdot\lambda \in Y(C_H(K)) \subseteq Y(M)$, and we are done.
\end{proof}

Our final proposition in this subsection is a strengthening of Proposition~\ref{prop:relGcrvsGir}(i) and \cite[Prop.\ 5.7]{BMRT}.
Before we begin the build-up to the result, we note that it would be possible to adapt the proof of \cite[Prop.\ 5.7]{BMRT}
to prove our new result in positive characteristic.
However, there are problems in characteristic zero: \cite[Prop.\ 5.7]{BMRT} is easy to prove in characteristic zero
because of the nice characterization of $G$-complete reducibility in this case (a subgroup is $G$-cr if and only if it is reductive),
but we have no such characterization of relative $G$-complete reducibility in characteristic zero, so we need a new proof.
The following preparatory work, which is motivated by the notion of a generic tuple from Section \ref{subsec:generictuple}, allows us to present a proof which is valid in any
characteristic.

Let $W$ be a finite-dimensional vector space over $k$ and let $n\in \NN$.
We have an action of $\GL_n$ on $W^n$ given by
$$ A\cdot (v_1,\ldots, v_n)= (v_1',\ldots, v_n'), $$
where
$$ v_i'= \sum_{j=1}^n a_{ij}v_j $$
and $a_{ij}$ is the $ij$-component of $A$.
This action of $\GL_n$ commutes with the
diagonal $\GL(W)$-action on $W^n$.

Recall the notion of a stable point from Definition \ref{defn:stability};
and recall from Remark \ref{rem:stability} that if
$v$ is $G$-stable, then
\begin{equation*}
 \pi^{-1}(\pi(v))= G\cdot v,
\end{equation*}
where $\pi\colon V\ra V/\!\!/G$ is the canonical morphism.

\begin{lem}
\label{lem:Thetastab}
 Let ${\mathbf w}= (w_1,\ldots, w_n)\in W^n$ such that the $w_i$
are linearly independent.
Let $F$ be a subgroup of $\GL_n$ such that $F$ is a finite
extension of ${\rm
SL}_n$.  Then ${\mathbf w}$ is a stable point for the action
of $F$ on $W^n$.
\end{lem}

\begin{proof}
Since the $w_i$ are linearly independent, $C_F({\mathbf w})$ is trivial.
It remains to check that $F\cdot {\mathbf w}$ is closed.
Let $\lambda\in Y(F)$ such that
$\underset{a\to 0}{\lim}\,  \lambda(a)\cdot {\mathbf w}$ exists.
Choose $A\in \GL_n$ such that
$A\cdot \lambda$ takes values in the subgroup of $F$ of diagonal matrices.
Set ${\mathbf w}'=
A\cdot {\mathbf w}= (w_1',\ldots, w_n')$.
Then $\underset{a\to 0}{\lim}\,  (A\cdot \lambda)(a)\cdot
{\mathbf w}'$ exists. There exist $m_1,\ldots,m_n\in {\mathbb Z}$ such that for all
$a\in k^*$, we have
$$
(A\cdot \lambda)(a)\cdot(w_1',\ldots, w_n')=(a^{m_1}w_1',\ldots, a^{m_n} w_n').
$$
Since $F^0= {\rm SL}_n$, we must have $\sum_{i=1}^n m_i= 0$.
As $w_i'\neq 0$ and
$\underset{a\to 0}{\lim}\,  a^{m_i}w_i'$ exists for each $i$,
we must have $m_i\geq 0$ for each $i$.
This forces all the $m_i$ to be zero, and so $A\cdot \lambda$ is trivial, hence
$\lambda$ is trivial.
We conclude that $F\cdot {\mathbf w}$ is closed, by the
Hilbert-Mumford Theorem.
\end{proof}

Now suppose that ${\rm dim}\,W=n$ and ${\mathbf w}\in W^n$ consists of a basis for $W$.  Let $g\in\GL(W)$. Then $g\cdot {\mathbf w}$ also consists of
a basis for $W$. Hence there
exists a unique element $A(g,{\mathbf w})\in {\rm GL}_n$ such that
\begin{equation}
\label{eqn:A(g)}
 g\cdot {\mathbf w}= A(g,{\mathbf w})\cdot {\mathbf w}.
\end{equation}
Let ${\mathbf w}'\in W^n$
be another basis for $W$.  We can write
${\mathbf w}'= A\cdot {\mathbf w}$ for some $A\in {\rm GL}_n$.
Using the fact that the actions of $\GL(W)$ and $\GL_n$ commute,
a straightforward calculation shows that
$$ A(g,{\mathbf w}')= AA(g,{\mathbf w})A^{-1}. $$
Hence ${\rm det}\,A(g,{\mathbf w})$ is independent of ${\mathbf w}$.
Moreover, if
$g'\in \GL(W)$, then we have
$$ A(g'g,{\mathbf w})= A(g',g\cdot {\mathbf w})A(g,{\mathbf w}).$$
It follows that the map
\begin{equation*}
\GL(W)\ra k^*, \ g\mapsto \det A(g,{\mathbf w})
\end{equation*}
is a homomorphism and is independent of ${\mathbf w}$.

We can now give our strengthening of Proposition~\ref{prop:relGcrvsGir}(i)
and \cite[Prop.~5.7]{BMRT}; note our new result
is significantly stronger, since we do not require
$K$ to be a subgroup of $C_G(S)$.
In the case $H=G$, Proposition \ref{prop:SnormalizesK}
and Corollary \ref{cor:SnormalizesK}
identify instances where
$G$-complete reducibility implies or is implied by relative $G$-complete
reducibility with respect to a proper subgroup.

\begin{prop}
\label{prop:SnormalizesK}
Let $K$, $H$ and $S$ be subgroups of $G$ such that $H$ is reductive and
$S$ normalizes $H$ and $K$. Put $M=C_H(S)^0$.
\begin{enumerate}[{\rm (i)}]
\item Suppose that $S$ is reductive and $HS$-completely reducible. Then $M$ is
reductive. Moreover, $K$ is relatively $G$-completely reducible
with respect to $H$ if it is relatively
$G$-completely reducible
with respect to $M$.
\item Suppose that
 \begin{enumerate}[{\rm (a)}]
 \item $\hh/\cc_\hh(S)$ does not have any trivial $S$-composition factors;
 \item $\cc_\hh(S)=\Lie C_H(S)$;
 \item $M$ is reductive.
 \end{enumerate}
Then $K$ is relatively $G$-completely reducible
with respect to $M$ if it is relatively $G$-completely reducible
with respect to $H$.
\end{enumerate}
\end{prop}

\begin{proof}
(i).\ Suppose that $S$ is reductive and $HS$-cr
(note that $HS$ is reductive, because $H$ and $S$ are). Then $M$ is reductive,
since $C_H(S)$ is $H$-cr by \cite[Thm.~5.4(a)]{BMRT}.
Suppose that $K$ is not relatively $G$-cr with respect to $H$.
We show that $K$ is not relatively $G$-cr with respect to $M$.
Let $\Omega(K)$ be the class of optimal cocharacters for $K$ in $H$
and let $P(K)$ be the corresponding
optimal destabilizing R-parabolic
subgroup for $K$ with respect to $H$ (Definition \ref{defn:optpar}).
By Theorem~\ref{thm:optpar}(ii), since $S \subseteq N_G(H) \cap N_G(K)$,
we have $S \subseteq P(K)$,
and hence $S$ normalizes $P(K) \cap H$.
By \cite[Lem.~5.1]{BMRT} (applied to the reductive group $HS$), $S$ normalizes an R-Levi subgroup of $P(K) \cap H$,
and by \cite[Thm.~4.5(iv)]{GIT}, this subgroup is of the form $L_\lambda(H)$ for
a unique $\lambda \in \Omega(K)$.
But $S$ acts on $\Omega(K)$,
by Theorem \ref{thm:optpar}(ii),
so $S$ must fix $\lambda$, and we
have $S \subseteq L_\lambda = C_G(\lambda(k^*))$.
Hence $\lambda \in Y(M)$.  If $K$ is contained in
an R-Levi subgroup of $P_\lambda(M)$,
then $K\subseteq L_{u\cdot \lambda}(M)$ for some $u\in R_u(P_\lambda(M))$.  But then $u\in R_u(P_\lambda(H))$, so $K$ is contained in $L_{u\cdot \lambda}(H)$, an R-Levi subgroup of $M$, which is impossible.  We conclude that $K$
is not relatively $G$-cr with respect to $M$ either.

(ii).\
After embedding $G$ in some $\GL_m$, we may assume that $G=\GL_m$.
Let $E$ be the linear span of $K$ in $\Mat_m$.
Let ${\bf e}\in K^n$ for some $n$ be a basis for the associative algebra $E$.
Then it follows from
Theorem \ref{thm:relLieGcr}
that $K$ is relatively $G$-cr with respect to $H$ if and only if
$H\cdot{\bf e}$ is closed, and likewise if we replace $H$ by $M$.  Before we proceed further, we briefly give the main idea of the proof.  We want to apply \cite[Thm.~5.4(b)]{BMRT} to ${\mathbf e}$.  We cannot do this directly,
because $S$ does not centralize $K$ --- it only normalizes $K$ --- and hence ${\mathbf e}$ need not be centralized by $S$.  The point of the argument below, and of Lemma \ref{lem:Thetastab}, is to allow us to replace ${\mathbf e}$ with $\pi{(\mathbf e})$, which {\bf is} centralized by $S$;
here $\pi:(\Mat_m)^n\to (\Mat_m)^n/\!\!/F$ is the canonical projection, with $F$ to be defined below.

Since $S$ normalizes $K$, we
have a homomorphism $\varphi:S\to\GL(E)$
and $E^n$ is an $S$-stable subset of $(\Mat_m)^n$ under the diagonal action.
Define $\psi:S\to k^*$ by $\psi(s)= {\rm det}\,A(\varphi(s),{\mathbf e})$, where $A(g,{\mathbf e})$ is as in Eqn.\ (\ref{eqn:A(g)}).
If $\psi(S)$ is finite, then we define $F\le\GL_n$ by $F=\det^{-1}(\psi(S))$.
Now assume that $\psi(S)=k^*$. Let $\Lambda\le k^*$ be the subgroup of all roots of
unity and for a positive integer $l$ let $\Lambda_l\le k^*$ be the group of $l$th roots
of
unity. Since $\Lambda$ is dense in $k^*$ and the inverse image of a dense subgroup
under a surjective homomorphism of algebraic groups is dense, we have that the ascending
chain $\psi^{-1}(\Lambda_{l!})_{l\ge1}$ of subgroups of $S$ is dense in $S$. By similar
arguments as in \cite[Prop.~3.7]{BMRT}, we can now replace $S$ by some
$\psi^{-1}(\Lambda_{l!})$ without changing $M$ and such that assumptions (a), (b) and
(c) still hold. Now we are again in the situation that $\psi(S)$ is finite and again
we define $F\le\GL_n$ by $F=\det^{-1}(\varphi(S))$. Let $\pi:(\Mat_m)^n\to (\Mat_m)^n/\!\!/F$ be the
canonical projection.

Now assume that $K$ is relatively $G$-cr with respect to $H$.
Let $\lambda\in Y(M)$
such that
${\mathbf e}':= \underset{a\to 0}{\lim}\, \lambda(a)\cdot {\mathbf e}$ exists.
Since
$H\cdot {\mathbf e}$ is closed, we have ${\mathbf e}'\in H\cdot {\mathbf e}$. So
$\pi({\mathbf e}')=\underset{a\to 0}{\lim}\, \lambda(a)\cdot\pi({\mathbf e})\in H\cdot \pi({\mathbf
e})$. By our choice of $F$ we have that for every $s\in S$, $A(\varphi(s),{\mathbf
e})\in F$. So $\pi({\mathbf e})$ is $S$-fixed. Now $M\cdot \pi({\mathbf e})$ is closed
in $H\cdot \pi({\mathbf e})$, by \cite[Thm.~5.4(b)]{BMRT}, so $\pi({\mathbf e}')\in
M\cdot \pi({\mathbf e})$. But then ${\mathbf e}'\in M\cdot {\mathbf e}$, by
Lemma~\ref{lem:Thetastab} and \cite[Cor.~3.5(ii)]{GIT}.
So $M\cdot {\mathbf e}$ is closed and thus $K$ is
relatively $G$-completely reducible with respect to $M$.
\end{proof}

We get the following analogue of \cite[Cor.\ 3.21]{BMR}.

\begin{cor}
\label{cor:SnormalizesK}
Let $G$, $H$, $K$, $S$, and $M = C_H(S)^0$  be as in
Proposition \ref{prop:SnormalizesK}.
Suppose that $S$ is linearly reductive.
Then $K$ is relatively $G$-completely reducible with respect to $H$
if and only if it is relatively $G$-completely reducible with respect to $M$.
\end{cor}

\begin{proof}
 Since $S$ is linearly reductive, $S$ is $HS$-cr \cite[Lem.\ 2.6]{BMR} and conditions (a)--(c) of Proposition \ref{prop:SnormalizesK}(ii) all hold
\cite[Lem.\ 4.1, Prop.\ 10.1.5]{rich1}.  The result now follows from Proposition \ref{prop:SnormalizesK}.
\end{proof}

\begin{rem}
 In \cite[Thm.~5.8]{BMR} and \cite{martin2}, there is a similar construction with the matrices $A(g,{\mathbf w})$ replaced by permutation matrices.  The construction above extends this idea.
\end{rem}

\section{Rationality questions}
\label{sec:rationality}

There is an obvious way to extend the notion of relative
$G$-complete reducibility
by considering non-algebraically closed fields.
Throughout this section, $k$ denotes any field
and we assume that $G$ is a reductive $k$-group.
Furthermore, we assume that $H$ is a reductive $k$-defined subgroup of $G$.
We let $k_s$ denote the separable closure of $k$,
and $\ovl k$ the algebraic closure of $k$.
We denote the Galois group ${\rm Gal}(k_s/k)={\rm Gal}(\ovl k/k)$ by $\Gamma$.
Algebraic groups and varieties will always be defined
over $\ovl k$ and points will always be $\ovl k$-defined points.
If $V$ is a $k$-variety and $k_1/k$ is an algebraic extension, then we denote the set of $k_1$-points of $V$ by $V(k_1)$.  We call elements of $V(k_s)$
\emph{separable points}.
Note that $\Gamma$ acts on $V=V(\ovl k)$. A closed subvariety $W$ of $V$ is defined over $k$
if and only if it contains a $\Gamma$-stable set of separable points of $V$ which is dense in $W$;
see \cite[Thm.~A.14.4]{Bo}. The set of $k$-defined cocharacters of a $k$-group $K$ is denoted $Y_k(K)$.  We say that a $G$-variety $V$ is defined over $k$ if both $V$ and the action of $G$ on $V$ are defined over $k$.

We begin with the definition of relative $G$-complete reducibility over $k$:
\begin{defn}
\label{defn:relativelyk}
Let $K$ be a subgroup of $G$.
We say that $K$ is
\emph{relatively $G$-completely reducible over $k$ with respect to $H$} if
for every $\lambda \in Y(H)$ such that $P_\lambda$ is $k$-defined and $K\subseteq P_\lambda$,
there exists $\mu \in Y(H)$ such that
$P_\lambda = P_\mu$, $L_\mu$ is $k$-defined and $K \subseteq L_\mu$.
\end{defn}

In order to deal with the definition of relative $G$-complete reducibility over $k$, we need some more detailed information about R-parabolic subgroups defined over $k$.

\begin{lem}\label{lem:k-defined}
Let $\lambda\in Y(H)$.
If $P_\lambda$ is $k$-defined, then $P_\lambda(H)$ is $k$-defined. The analogous
assertions hold for $L_\lambda$ and $R_u(P_\lambda)$.
\end{lem}

\begin{proof}
Thanks to \cite[Prop.~12.1.5]{spr2},
$P_\lambda(H) = P_\lambda \cap H$ is $k$-defined
if $ \Lie P_\lambda(H) = \Lie P_\lambda \cap \Lie H$; similarly
for $L_\lambda$ and $R_u(P_\lambda)$.
The result now follows from
Lemma \ref{lem:liealgebrasofRpars} applied to $G$ and $H$.
\end{proof}

\begin{rem}
 It follows from Lemma \ref{lem:k-defined} and \cite[Lem.~2.5]{GIT} that if $\lambda\in Y_k(H)$, then $P_\lambda$, $L_\lambda$, $R_u(P_\lambda)$, $P_\lambda(H)$, $L_\lambda(H)$ and $R_u(P_\lambda(H))$ are $k$-defined.
\end{rem}

\begin{lem}\ 
\label{lem:R-parabolic}
Let $\lambda\in Y(H)$ and assume that $P_\lambda$ is $k$-defined. Then there
exists $\mu\in Y_k(H)$ such that $P_\lambda\subseteq P_\mu$ and $P_\lambda^0=P_\mu^0$.
\end{lem}

\begin{proof}
This follows from the proof of \cite[Lem.~2.5(ii)]{GIT},
replacing the maximal torus of $P_\lambda$ with a maximal
torus of  $P_\lambda(H)$.
\end{proof}

\begin{rem}
It is not true that any $k$-defined R-parabolic subgroup of $G$
stems from a cocharacter defined over $k$;
see \cite[Rem.~2.4]{GIT}.

\end{rem}

The next lemma gives a version of Lemma~\ref{lem:relativelystrong}(i) over $k$.  Note that any two $k$-defined R-Levi subgroups of a $k$-defined R-parabolic subgroup are $R_u(P)(k)$-conjugate \cite[Lem.~2.5(iii)]{GIT}.

\begin{lem}
\label{lem:relativelystrongk}
Let $\lambda,\mu\in Y(H)$ such that $P_\lambda=P_\mu$ and
let $u$ be the element of $R_u(P_\lambda(H))$ such that
$uL_\lambda(H)u\inverse=L_\mu(H)$.
Then $uL_\lambda u\inverse=L_\mu$.
Furthermore, if $G$, $H$, $P_\lambda$, $L_\lambda$ and $L_\mu$ are $k$-defined,
then $u \in R_u(P_\lambda(H))(k)$.
\end{lem}

\begin{proof}
The first part of this is Lemma \ref{lem:relativelystrong}.
The final assertion follows from \cite[Lem.\ 2.5(iii)]{GIT}
and the fact that $R_u(P_\lambda(H))(k)=R_u(P_\lambda(H))\cap R_u(P_\lambda)(k)$.
\end{proof}

\begin{cor}
\label{cor:relativelystrong}
Let $\lambda\in Y_k(H)$ and let $\mu\in Y(H)$ such that $P_\lambda=P_\mu$
and $L_\mu$ is $k$-defined.
Then there exists $\nu\in Y_k(H)$ such that $P_\lambda=P_\nu$ and $L_\mu=L_\nu$.
\end{cor}

\begin{proof}
By Lemma~\ref{lem:relativelystrongk}, there exists $u\in R_u(P_\lambda(H))(k)$ such that
$L_{u\cdot\lambda}=uL_\lambda u\inverse=L_\mu$, so we can take $\nu=u\cdot\lambda$.
\end{proof}

We can now show that when discussing relative $G$-complete reducibility over $k$,
it suffices to consider R-parabolic subgroups of the form $P_\lambda$ with $\lambda \in Y_k(H)$,
rather than all $k$-defined R-parabolic subgroups stemming from a cocharacter of $H$.

\begin{lem}
Let $K$ be a subgroup of $G$.
Then $K$ is relatively $G$-completely reducible over $k$ with respect to $H$ if and only if
for every $\lambda \in Y_k(H)$ such that $K$ is contained in $P_\lambda$,
there exists $\mu \in Y_k(H)$ such that
$P_\lambda = P_\mu$ and $K \subseteq L_\mu$.
\end{lem}

\begin{proof}
 This follows from the proof of \cite[Lem.~5.2]{GIT}, using Corollary \ref{cor:relativelystrong} in place of \cite[Cor.~2.6]{GIT}.
\end{proof}

In order to generalize Theorem \ref{thm:mainthm} to arbitrary fields,
we need a notion of a ``closed orbit'' for a group $M(k)$ of $k$-points
of a reductive $k$-group $M$ acting on a $k$-variety.
As we shall show, the correct notion for us is given by the
following definition, see \cite[Def.~3.8]{GIT}:

\begin{defn}\label{defn:cocharclosed}
Let $M$ be a reductive $k$-group and let $V$ be an $M$-variety defined over $k$.
Let $v \in V$. We say that the $M(k)$-orbit $M(k)\cdot v$ is
\emph{cocharacter-closed over $k$} if for
any $\lambda \in Y_k(M)$ such that
$v':= \lim_{a\to 0}\lambda(a)\cdot v$ exists, $v'$ is $M(k)$-conjugate
to $v$.
Note that it follows from the Hilbert-Mumford Theorem that
$M\cdot v$ is closed if and only if $M(\ovl{k})\cdot v$
is cocharacter-closed over $\ovl{k}$.
\end{defn}

For the proof of Theorem~\ref{thm:orbitcritk} we need
the following two extensions of Theorem \ref{thm:Ruconj}
to non-algebraically closed fields. The first result is
\cite[Thm.~3.3]{GIT}. Here we require the field to be perfect.

\begin{thm}
\label{thm:Ruconjk}
Suppose $k$ is a perfect field.
Let $M$ be a reductive $k$-group and let
$V$ be an $M$-variety defined over $k$.
Let $v\in V$ and let $\lambda\in Y_k(M)$ such
that $v':=\underset{a\to 0}{\lim} \lambda(a)\cdot v$
exists and is $M(k)$-conjugate to $v$.
Then $v'$ is $R_u(P_\lambda(M))(k)$-conjugate to $v$.
\end{thm}

The second result is \cite[Thm.~3.10]{GIT}.
Here we require $M$ to be connected and the two assertions
need to be quantified over all $k$-defined cocharacters from $M$.

\begin{thm}\label{thm:git1.4}
Let $M$ be a connected reductive $k$-group and let
$V$ be an $M$-variety defined over $k$.
Let $v\in V$. Then the following are equivalent:
\begin{enumerate}[{\rm (i)}]
\item $M(k)\cdot v$ is cocharacter-closed over $k$;
\item for all $\lambda \in Y_k(M)$
such that $v':=\underset{a\to 0}{\lim} \lambda(a)\cdot v$ exists,
$v'$ is $R_u(P_\lambda(M))(k)$-conjugate to $v$.
\end{enumerate}
\end{thm}

Using the preceding discussion, we can now extend parts of
Theorem \ref{thm:orbitcrit} to non-algebraically closed fields.
The final assertion in part (iii) is the desired extension of
Theorem \ref{thm:mainthm} to arbitrary fields.
Note that this equivalence also generalizes \cite[Thm.~5.9]{GIT}.

\begin{thm} Let $H$ be a reductive $k$-defined subgroup of $G$.
\label{thm:orbitcritk}
\begin{enumerate}[{\rm (i)}]
\item Let $n\in {\mathbb N}$, let $\tuple{k}\in G^n$ and let
$\lambda\in Y_k(H)$ such that
$\tuple{m}=\lim_{a\ra 0}\lambda(a)\cdot \tuple{k}$ exists.
Then the following are equivalent:
\begin{enumerate}[{\rm (a)}]
\item $\tuple{m}$ is $R_u(P_\lambda(H))(k)$-conjugate to $\tuple{k}$;
\item there exists $\mu \in Y_k(H)$ such that
$P_\lambda = P_\mu$ and $\mu(k^*)$ fixes $\tuple{k}$.
\end{enumerate}
   If $k$ is perfect, then (a) and (b) are also equivalent to the following:
\begin{enumerate}[{\rm (c)}]
\item $\tuple{m}$ is $H(k)$-conjugate to $\tuple{k}$.
\end{enumerate}
\item Let $K$ be a subgroup of $G$ and let $\lambda\in Y_k(H)$.
Suppose $K\subseteq P_\lambda$ and set $M= c_\lambda(K)$.
Then the following are equivalent:
\begin{enumerate}[{\rm (a)}]
\item $M$ is $R_u(P_\lambda(H))(k)$-conjugate to $K$;
\item $K \subseteq L_\mu$ for some
$\mu\in Y_k(H)$ such that $P_\lambda = P_\mu$.
\end{enumerate}
\item Let $K$, $\lambda$ and $M$ be as in (ii) and let $\tuple{k}\in K^n$
be a generic tuple of $K$.
    Then the assertions in (i) are equivalent to those in (ii).
    Furthermore, $K$ is relatively $G$-completely reducible  over $k$
with respect to $H$ if and only if $H^0(k)\cdot\tuple{k}$ is
cocharacter-closed over $k$.
\end{enumerate}
\end{thm}

\begin{proof}
(i) and (ii).
The equivalence of (a) and (b) in both cases follows from Lemmas~\ref{lem:Ruconj} and \ref{lem:relativelystrongk}.  The equivalence of part (a) and part (c) in (i)
follows from Theorem \ref{thm:Ruconjk}.

(iii). The equivalence of (i)(b) and (ii)(b) follows from Lemma \ref{lem:generictuple}(ii),
so (i) and (ii) are equivalent.
For the final assertion of (iii) we can assume $H$ is connected, by Remark~\ref{rem:relativelystrong}(i).  
Now we use Theorem~\ref{thm:git1.4}
and the equivalence of (i) and (ii).
\end{proof}

The following result extends \cite[Thm.~5.11]{GIT}
to our setting of relative $G$-complete reducibility.
The proof is completely analogous to the proof given in \emph{loc.~cit}.

\begin{thm}
\label{thm:serresquestion}
Suppose $k_1/k$ is a separable extension of fields.
Let $K$ be a $k$-defined subgroup of $G$.
If $K$ is relatively $G$-completely reducible over $k_1$ with respect to $H$, then
$K$ is relatively $G$-completely reducible over $k$ with respect to $H$.
\end{thm}

We end this section with a converse of Theorem~\ref{thm:serresquestion}
in case $k$ is perfect.
For this one needs an extension to arbitrary $k$ of the
results on optimal parabolic subgroups from Section~\ref{sec:optpars}.
We briefly indicate how to set this up.

First, we adapt some of the notation from \cite[Sec.~4, Sec.~5]{GIT}
to the relative setting over $k$.
Suppose $K$ is a subgroup of $G$
and suppose $\lambda \in Y_k(H)$ is such that $K \subseteq P_\lambda$.
Set $M := c_\lambda(K)$ and, as in
Section~\ref{sec:optpars}, let $S_n(M) = \overline{H\cdot M^n}$.
Then $K^n$ is uniformly $S_n(M)$-unstable over $k$ for the action of $H$
on $G^n$ in the sense of \cite[Def.~4.2]{GIT}. Any $G(k)$-invariant norm
on $Y_k(G)$ in the sense of \cite[Def.~4.1]{GIT} restricts to an $N_{G(k)}(H)$-invariant
norm on $Y_k(H)$; let $\left\|\,\right\|$ be such a norm.
Then \cite[Def.~4.4]{GIT} provides an \emph{optimal class}
$\Omega(K^n,S_n(M),k) \subseteq Y_k(H)$ of $k$-defined cocharacters of $H$.

With these preliminaries in hand, it is straightforward to derive
analogues of \cite[Thm.~5.16]{GIT}, \cite[Def.~5.17]{GIT} and \cite[Thm.~5.18]{GIT}
in the relative setting.
The first result is derived as follows:
let $N$ be the closure in $G$ of $N_{G(k_s)}(H)$.
Then $N$ is a $k$-defined closed
subgroup of $N_G(H)$, $H$ is a normal subgroup of $N$, and $N(k)=N_{G(k)}(H)$.
Now apply \cite[Thm.~4.5]{GIT} with $(G',G,V,X,S)=(N,H,G^n,K^n,S_n(M))$.
For the second result, one needs to choose an $N$-invariant $k$-defined norm
on $Y(H)$.

\begin{thm}
\label{thm:gcroptoverk}
Suppose $k_1/k$ is an algebraic extension of perfect fields.
Let $K$ be a $k$-defined subgroup of $G$.
If $K$ is relatively $G$-completely reducible
over $k$ with respect to $H$, then
$K$ is relatively $G$-completely reducible over $k_1$ with respect to $H$.
\end{thm}
\begin{proof}
By Theorem~\ref{thm:serresquestion}, we may assume that $k_1=\ovl k$.
Suppose $K$ is not relatively $G$-completely reducible with respect to $H$.
Then $K$ is not contained in any R-Levi subgroup of the optimal
destabilizing R-parabolic subgroup $P(K)$ with respect to $H$.
Now $P(K)$ is $k$-defined by the analogue of \cite[Thm.~5.18(ii)]{GIT} 
so $K$ is not $G$-completely reducible over $k$ with respect to $H$.
\end{proof}

\begin{rem}
We do not know whether 
Theorem \ref{thm:gcroptoverk} holds for an arbitrary separable algebraic field extension $k_1/k$.  
For further discussion, see \cite[Ex.~5.21]{GIT} and \cite[Rem 1.2]{BMR10}.
\end{rem}

\begin{rem}
 The results in this section have obvious counterparts for Lie algebras.  We leave the details to the reader.
\end{rem}

\section{Examples and counterexamples}
\label{sect:ex}
\subsection{Relative $\GL(V)$-complete reducibility}
\label{sub:relGLV-cr}

Now we investigate the concept of relative complete
reducibility in case the ambient reductive group is a general linear
group.

Let $V$ be a finite-dimensional $k$-vector space and set $G = \GL(V)$.
Recall that if $K$ is a subgroup of $G$, then $K$ is $G$-cr if and only if $V$
is a semisimple module for $K$.
In this subsection we give an analogous interpretation for relative
$G$-complete reducibility with respect to a smaller general linear group $H$ inside
$G$.

\begin{prop}
\label{prop:relativemodules}
Let $U$ be a subspace of $V$, and pick a direct complement $\tilde U$ to $U$.
Let $H = \GL(U) \subseteq G$ embedded in the obvious way.
Let $K$ be a subgroup of $G = \GL(V)$.
Then $K$ is relatively $G$-completely reducible with respect to $H = \GL(U)$
if and only if the following two conditions hold:
\begin{enumerate}[{\rm(i)}]
\item every $K$-submodule of $V$ contained in $U$ has a $K$-complement in $V$
containing $\tilde U$;
\item every $K$-submodule of $V$ containing $\tilde U$ has a
$K$-complement in $V$ contained in $U$.
\end{enumerate}
\end{prop}

The proof is straightforward and we leave it as an exercise for the reader.

\begin{rems}
(i).
Note that conditions (i) and (ii) in Proposition \ref{prop:relativemodules}
are ``dual'' to each other in the sense
that the complement in (i) is a submodule of the form in (ii), and vice versa.
One considers only those decompositions of $V$ as a direct sum of $K$-submodules
that are compatible with the fixed
decomposition $V = U \oplus \tilde U$
(even though $U$ and $\tilde U$ are not required to be $K$-stable!).
Even in the very special case 
considered in Proposition \ref{prop:relativemodules} the concept of a
relatively completely reducible subgroup of $\GL(V)$ appears to be new.

(ii).
Note that for $H = G$ in Proposition \ref{prop:relativemodules}
we recover the fact that $K$ is $\GL(V)$-completely reducible if and only
if $V$ is a semisimple $K$-module. 

(iii).
Fixing the complementary subspace $\tilde U$ to $U$ in
Proposition \ref{prop:relativemodules} fixes the embedding of $H$ in $G$.
Note that this is crucial, as the result depends on the choice of $\tilde U$.
To see this, consider the special case that $K$ acts completely reducibly on
$V$ and $U$ is a $K$-submodule of $V$.
Then it follows from Proposition \ref{prop:relativemodules} that
$K$ is relatively $G$-cr with respect to $H$ if and only if $\tilde U$ is
also a $K$-submodule.
\end{rems}

We can refine Proposition \ref{prop:relativemodules} to show more
accurately how the conclusion depends on the structure of $V$ as a $K$-module,
and how this interacts with the subspace $U$.
To do this, we first define two operations on the set of subspaces of $V$.
Firstly, for any subspace $W \subseteq V$, let $\sigma_K(W)$ be the
$K$-submodule of $V$ generated by the submodules contained in $W$;
note $\sigma_K(W) \subseteq W$.
Secondly, let $\iota_K(W)$ denote the smallest $K$-submodule of
$V$ containing $W$ (i.e., the intersection of all such submodules).
Then we have the following result.

\begin{prop}\label{prop:relativemodrefined}
Let $G$, $U$, $\tilde U$, $H$ and $K$ be as in Proposition \ref{prop:relativemodules}.
Then $K$ is relatively $G$-completely reducible with respect to $H$ if and only
if the following two conditions hold:
\begin{enumerate}[{\rm(i)}]
\item $\sigma_K(U)$ is a completely reducible $K$-module;
\item $V = \sigma_K(U) \oplus \iota_K(\tilde U)$.
\end{enumerate}
\end{prop}

\begin{proof}
Suppose $K$ is relatively $G$-cr with respect to $H$. We first show (i).
Let $W$ be a $K$-submodule of $\sigma_K(U)$.
Then $W \subseteq U$, so there exists a $K$-submodule $\tilde W \supseteq \tilde U$
with $V= W \oplus \tilde W$,
by condition (i) of Proposition \ref{prop:relativemodules}.
Then $\tilde W \cap \sigma_K(U)$ is a $K$-complement to $W$ in $\sigma_K(U)$.

We now show (ii).
By condition (i) of Proposition \ref{prop:relativemodules}, since
$\sigma_K(U)$ is a $K$-submodule contained in $U$, there exists a $K$-complement
to $\sigma_K(U)$ containing $\tilde U$.
Since $\iota_K(\tilde U)$ is the smallest such $K$-submodule, we must have
$\sigma_K(U) \cap \iota_K(\tilde U) = \{0\}$.
By condition (ii) of Proposition \ref{prop:relativemodules}, since
$\iota_K(\tilde U)$ is a $K$-submodule containing $\tilde U$, there exists a $K$-complement
to $\iota_K(\tilde U)$ contained in $U$.
Since $\sigma_K(U)$ is the largest such $K$-submodule, we must have
$\sigma_K(U) + \iota_K(\tilde U) = V$.

Now suppose (i) and (ii) hold.
If $W$ is a $K$-submodule of $V$ contained in $U$, then
$W \subseteq \sigma_K(U)$, by definition.
By (i), there exists a
$K$-complement $W'$ to $W$ in $\sigma_K(U)$.
Then $W' \oplus \iota_K(\tilde U)$ is a $K$-complement to
$W$ in $V$ containing $\tilde U$, by (ii), giving condition (i) of Proposition \ref{prop:relativemodules}.
If $W$ is a $K$-submodule of $V$ containing $\tilde U$, then
$W \supseteq \iota_K(U)$, by definition.
Let $W'$ be a $K$-complement to $W \cap \sigma_K(W)$ in $\sigma_K(W)$.
Then $W'$ is a $K$-complement to $W$ in $V$ contained in $U$,
giving condition (ii)
of Proposition \ref{prop:relativemodules}.
\end{proof}

The following result gives a necessary condition for a subgroup to be
relatively $G$-cr with respect to a Levi subgroup of $G$
in terms of the corresponding direct sum decomposition of $V$:
\begin{cor}
\label{gl:cor1}
Suppose $L$ is a Levi subgroup of $G = \GL(V)$, with corresponding
decomposition $V = U_1 \oplus \cdots \oplus U_s$.
Let $K$ be a subgroup of $G$.
Then if $K$ is relatively $G$-completely reducible with respect to $L$, the following two conditions
hold for each $i$:
\begin{enumerate}[{\rm(i)}]
\item every $K$-submodule of $V$ contained in $U_i$
has a $K$-complement containing $\bigoplus_{j \neq i} U_j$;
\item every $K$-submodule of $V$ containing $\bigoplus_{j \neq i} U_j$
has a $K$-complement contained in $U_i$.
\end{enumerate}
\end{cor}

\begin{proof}
We have $L = \GL(U_1) \times \cdots \times \GL(U_s)$.  If $K$ is relatively $G$-cr with respect to $L$ then Proposition \ref{prop:normalrelative} implies that $K$ is relatively $G$-cr with respect to $\GL(U_i)$ for each $i$.
The result now follows from Proposition \ref{prop:relativemodules}.
\end{proof}

\subsection{More examples and counterexamples}
\label{sub:exandnot}

\begin{rem}
\label{rem:reductive}
We noted in Remark \ref{rem:relativelystrong}(iv) that in
general a relatively $G$-cr subgroup need not be reductive.
However, one can ensure that it is reductive under suitable conditions.
For example,
if $H$ is a maximal rank reductive subgroup of $G$ and
$K$ is a subgroup of $G$ which
is relatively $G$-completely reducible with respect to $H$
and which is normalized by a maximal torus of $H$, then
$K$ is reductive.
The special case of this result when $H = G$ is due to Serre, \cite[Prop.\ 4.1]{serre2},
and one can follow his argument to prove the result in the relative setting.
We leave the details to the reader.
\end{rem}

We noted at the beginning of Section \ref{sub:newfromold}
that the direct analogue of \cite[Thm.~3.10]{BMR},
namely that a normal subgroup $N$ of a subgroup $M$ of
$G$ is relatively $G$-cr with
respect to some reductive subgroup $H$ of $G$ provided $M$ is,
fails in general. We now present two examples which demonstrate this failure.

\begin{exmp}
\label{exmp:relative1}
Let $\Char k$ be arbitrary. Let $G =\GL_3$
and let $H$ be the image of $\SL_2$ embedded in $G$ by
$ A\mapsto
\left(
 \begin{array}{c|c}
  1 & 0  \\
  \hline
  0 & A \\
   \end{array}
\right)$.
Let $T$ be the standard maximal torus of $G$ consisting of the
diagonal matrices in $G$. Let $\alpha$ and $\beta$
be the standard simple roots of $G$ with respect to $T$ (corresponding respectively to the $(2,3)$- and $(1,2)$-entries of matrices in $G$).
Let $B$ be the Borel subgroup of $G$ consisting of upper triangular
matrices in $G$.
Put
$K = U_{\beta}U_{\alpha+\beta}$.
Then $K$ is relatively $G$-cr with respect to $H$, but the normal subgroups
$U_{\beta}$ and $U_{\alpha+\beta}$ of $K$ are not.


Take for example $N = U_{\alpha + \beta}$ and define $\lambda\in Y(H)$ by
$\lambda(a)= \diag(1,a,a^{-1})$.
Then $N\subseteq P_\lambda$,
%
which is a Borel subgroup of $G$.
Clearly, $N$ is not contained in any Levi subgroup of $P_\lambda$,
since $N\subseteq R_u(P_\lambda)$. So $U_{\alpha+\beta}$ is not
relatively $G$-cr with respect to $H$.
The argument for $U_{\beta}$ is similar; replace $\lambda$ with $-\lambda$.

Now we show
that $K$ is relatively $G$-cr with respect to $H$ by showing
that $\lambda\in Y(H)$ and $H\subseteq P_\lambda$ implies that
$\lambda = 0$. Let $\lambda\in Y(H)$. We can find $h\in H$ such
that $\mu:=h\cdot\lambda$ is in diagonal form. Then we have for
$x \in K$ that $\underset{a \ra 0}{\lim}\, \lambda(a)x\lambda(a)^{-1}$
exists if and only if $\underset{a \ra 0}{\lim}\,  \mu(a)h x h^{-1}\mu(a)^{-1}$
exists. Since $K$ is stable under conjugation by $H$, we may now assume that
$\lambda$ is in diagonal form, that is,
$\lambda(a) = \diag(1,a^n, a^{-n})$
for some integer $n$. It is now straightforward to show that $n=0$.
\end{exmp}

The next example shows that a connected
reductive normal subgroup of a connected reductive group that is
relatively $G$-cr with respect to $H$ need not be relatively $G$-cr with respect to $H$.

\begin{exmp}
\label{exmp:relative2}
Suppose $p=2$.  Let $V_1$ and $V_2$ be copies of the vector space $k^2$ and let
$H_1= H_2= k^*\times {\rm SL}_2$.
We identify $H_1$ and $H_2$ with subgroups of $H_1 \times H_2$.
Define $\delta_i\in Y(H_i)$ by $\delta_i(x)= (x,I)$, where $I$ is the identity matrix in
${\rm SL}_2$.
Define an action of $H_1\times H_2$ on $V_1\oplus V_2$ by
\[
((x_1,A_1),(x_2,A_2))\cdot (v_1,v_2)= (x_2^{-1}A_1v_1, x_1^{-1}A_2v_2),
\]
where $A_iv_i$ denotes the usual matrix product.
We choose an embedding of $(H_1\times H_2)\ltimes (V_1\oplus V_2)$ inside a reductive group $G$.
Let $\pi_i\colon H_i\times V_i\ra H_i$ be the canonical projection.
Let $M_i$ be the copy of ${\rm SL}_2$ inside $H_i$.
We can choose a copy $N_i$ of ${\rm PGL}_2$ inside $M_i\ltimes V_i$ such that
$\pi_i(N_i)= M_i$ but $N_i$ is not $(H_i\ltimes V_i)$-conjugate to a subgroup of $H_i$;
to see this, note that the image of the adjoint representation of ${\rm SL}_2$ in ${\rm GL}_3$ lies in $[P,P]$, where $P$ is a maximal parabolic subgroup of ${\rm GL}_3$, and $[P,P]$ is isomorphic to $M_i\ltimes V_i$, so we can take $N_i$ to be this image.

Let $H= \{((x_1,A_1),(x_2,A_2))\in H_1\times H_2\,|\,x_2= x_1^{-1}\}$ and set $K=N_1N_2$.
Then $K$ is isomorphic to $N_1\times N_2$, since $N_1$ and $N_2$ commute with each other
and have disjoint Lie algebras. So $K$ is a connected reductive group and $N_1,N_2$
are connected normal subgroups of $K$.
We show that $K$ is relatively $G$-cr with respect to $H$ but $N_1$ and $N_2$ are not.

Let $\lambda\in Y(H)$ such that $K\subseteq P_\lambda$.
We can write $\lambda= \lambda_1+ \lambda_2$ where $\lambda_i\in Y(H_i)$.
We have $N_i\subseteq P_\lambda$.
Now $Z(H_i)\subseteq P_\lambda$ because $Z(H_i)$ centralizes the image of $\lambda$.
Since $Z(H_i)$ acts trivially on $M_i$ and acts non-trivially as multiplication by scalars on $V_i$,
it is clear that $M_i$ and the non-trivial subgroup $U_i$ of $V_i$ generated by the set
$\{\pi_i(n)\,|\,n\in N_i\}$ are both contained in $P_\lambda$.
Because $M_i$ does not lie in a proper parabolic subgroup of $H_i$,
we must have $\lambda= n_1\delta_1+ n_2\delta_2$ for some $n_1,n_2\in {\mathbb Z}$.
We must have $n_1+n_2=0$ by our choice of $H$.
We have $\lambda(x)v_1\lambda(x)^{-1}= x^{-n_2}v_1$ for $v_1\in V_1$ and
$\lambda(x)v_2\lambda(x)^{-1}= x^{-n_1}v_2$ for $v_2\in V_2$.
Since the non-trivial subgroups $U_1$ and $U_2$ lie in $P_\lambda$,
this forces $n_1=n_2=0$.
Hence $\lambda$ is the trivial cocharacter and we conclude that
$K$ is relatively $G$-cr with respect to $H$.

A similar argument shows that $N_1\subseteq P_\lambda$ and $U_1\subseteq R_u(P_\lambda)$, where $\lambda= \delta_1- \delta_2$.
It follows that $c_\lambda(N_1)= M_1$.  Now $N_1$ is not $H$-conjugate to $M_1$,
and this implies by Theorem \ref{thm:orbitcrit}(ii) that $N_1$ is not relatively $G$-cr with respect to $H$.
Similarly, $N_2$ is not relatively $G$-cr with respect to $H$.
\end{exmp}


\bigskip
{\bf Acknowledgements}:
The authors acknowledge the financial support of EPSRC Grant EP/C542150/1, Marsden Grant UOC0501 and 
the DFG-priority programme SPP1388 ``Representation Theory''.
Part of the research for this paper was carried out while the
authors were staying at the Mathematical Research Institute
Oberwolfach supported by the ``Research in Pairs'' programme.
Also, part of this paper was written during a stay of the three
first authors at the Isaac Newton Institute for Mathematical
Sciences during the ``Algebraic Lie Theory'' Programme in 2009.

We also wish to thank the referee for some helpful comments.

\end{document}